\documentclass[10pt]{article}
\usepackage{amssymb,amsmath,color}
\usepackage{amsthm}
\usepackage{mathrsfs}
\usepackage{fullpage}
\usepackage{mathrsfs}

\allowdisplaybreaks
\begin{document}
\def\pd#1#2{\frac{\partial#1}{\partial#2}}
\def\dfrac{\displaystyle\frac}
\let\oldsection\section
\renewcommand\section{\setcounter{equation}{0}\oldsection}
\renewcommand\thesection{\arabic{section}}
\renewcommand\theequation{\thesection.\arabic{equation}}

\def\Xint#1{\mathchoice
  {\XXint\displaystyle\textstyle{#1}}%
  {\XXint\textstyle\scriptstyle{#1}}%
  {\XXint\scriptstyle\scriptscriptstyle{#1}}%
  {\XXint\scriptscriptstyle\scriptscriptstyle{#1}}%
  \!\int}
\def\XXint#1#2#3{{\setbox0=\hbox{$#1{#2#3}{\int}$}
  \vcenter{\hbox{$#2#3$}}\kern-.5\wd0}}
\def\ddashint{\Xint=}
\def\dashint{\Xint-}

\newcommand{\be}{\begin{equation} \label}
\newcommand{\ee}{\end{equation}}
\newcommand{\beaa}{\begin{eqnarray}}
\newcommand{\bea}{\begin{eqnarray}\label}
\newcommand{\eea}{\end{eqnarray}}
\newcommand{\bas}{\begin{eqnarray*}}
\newcommand{\eas}{\end{eqnarray*}}
\newcommand{\nn}{\nonumber}
\newcommand{\Om}{\Omega}
\newcommand{\iO}{\int_\Omega}
\newcommand{\into}{\int_\Omega}
\newcommand{\intO}{\int_\Omega}
\newcommand{\cd}{\cdot}
\newcommand{\pa}{\partial}
\newcommand{\ep}{\varepsilon}
\newcommand{\Lp}{L^p(\O)}
\newcommand{\Lq}{L^q(\O)}
\newcommand{\de}{\delta}
\newcommand{\iT}{\int_{t_0}^t}
\newcommand{\Li}{L^\infty(\O)}
\newcommand{\Dtau}{(\frac{\Delta-1}{\tau})}
\newcommand{\1}{\delta_1}
\newcommand{\2}{\delta_2}
\newcommand{\nep}{n_\ep}
\newcommand{\cep}{c_\ep}
\newcommand{\uep}{u_\ep}
\renewcommand{\O}{\Omega}

\newcommand{\To}{\Rightarrow}
\newcommand{\blue}[1]{{\color{blue}#1}}
\newcommand{\red}[1]{{\color{red}#1}}
\newcommand{\green}[1]{{\color{green}#1}}
\newcommand{\nbar}{\bar{n}}
\newcommand{\Lin}{L^\infty(\O)}
\newcommand{\norm}[2][ ]{\|#2\|_{#1}}
\newcommand{\wstarto}{\stackrel{\star}{\rightharpoonup}}
\newcommand{\cred}{\color{red}}
\newcommand{\set}[1]{\{#1\}}

\newtheorem{thm}{Theorem}
\newtheorem{lem}{Lemma}[section]
\newtheorem{proposition}{Proposition}[section]
\newtheorem{dnt}{Definition}[section]
\newtheorem{remark}[lem]{Remark}
\newtheorem{cor}{Corollary}[section]
\allowdisplaybreaks

\title{Global classical solutions in chemotaxis(-Navier)-Stokes system with rotational flux term}
\author{
Xinru Cao\thanks{Insitute for Mathematical Sciences, Renmin University of China, Zhongguancun Str. 59, 100872 Beijing, China; email: caoxinru@gmail.com; Supported by the Fundamental Research Funds for the Central Universities, and the Research Funds of Renmin University of China (15XNLF21).}
}
{\small }
{\small }
\maketitle
\date{}
\begin{abstract}
\noindent
The coupled chemotaxis  fluid system
\bas
\left\{
\begin{array}{llc}
\displaystyle
n_t=\Delta n-\nabla\cdot(nS(x,n,c)\cdot\nabla c)-u\cdot\nabla n,
&(x,t)\in \Omega\times (0,T),\\[4pt]
\displaystyle
c_t=\Delta c-nc-u\cdot\nabla c , &(x,t)\in\Omega\times (0,T),\\[4pt]
\displaystyle
u_t=\Delta u-\kappa(u\cdot\nabla)u+\nabla P+n\nabla\phi , &(x,t)\in\Omega\times (0,T),\\[4pt]
\nabla\cdot u=0,&(x,t)\in\Omega\times (0,T),
\end{array}
\right.(\star)
\eas
is considered under the no-flux boundary conditions for $n,c$
and the Dirichlet boundary condition for $u$
on a bounded smooth domain $\O\subset\mathbb{R}^N$ ($N=2,3$), $\kappa=0,1$.
We assume that $S(x,n,c)$ is a matrix-valued sensitivity under
a mild assumption such that $|S(x,n,c)|<S_0(c_0)$ with some non-decreasing function $S_0\in C^2((0,\infty))$.
It contrasts the related scalar sensitivity case that
$(\star)$ does not possess the natural {\em gradient-like} functional structure. Associated estimates based on the natural functional seem no longer available.
In the present work, a global classical solution is constructed under a smallness assumption on $\|c_0\|_{L^\infty(\Om)}$ and moreover we obtain boundedness and large time convergence for the solution, meaning that small initial concentration of chemical forces stabilization.

\vspace{15pt}

\end{abstract}

\section{Introduction}
In this paper, we study the chemotaxis-Navier-Stokes system
\begin{align}\label{eq}
\begin{cases}
n_t=\Delta n-\nabla\cdot(nS(x,n,c)\cdot\nabla c)-u\cdot\nabla n,
&(x,t)\in \Omega\times (0,T),\\
\displaystyle
c_t=\Delta c-nc-u\cdot\nabla c , &(x,t)\in\Omega\times (0,T),\\
\displaystyle
u_t=\Delta u-\kappa(u\cdot\nabla)u+\nabla P+n\nabla\phi , &(x,t)\in\Omega\times (0,T),\\
\displaystyle
\nabla\cdot u=0, &(x,t)\in\Omega\times (0,T),\\
\displaystyle
\nabla c\cdot\nu=(\nabla n-S(x,n,c)\nabla c)\cdot\nu=0, u=0, &(x,t)\in\pa\O\times (0,T),\\
\displaystyle
n(x,0)=n_0(x), v(x,0)=v_0(x), u(x,0)=u_0(x), &x\in\O,
\end{cases}
\end{align}
where $T\in(0,\infty]$, $\kappa=0,1$, $\O\subset\mathbb{R}^N$ (N=2,3) is a bounded domain with smooth boundary and $\nu$ denotes the outward normal vector on $\pa\O$.
Here $S(x,n,c)=(s_{ij}(x,n,c))_{i,j\in\{1,2\}}$ is a matrix-valued function and $\phi\in W^{1,\infty}(\Om)$.\\

The PDE system of type (\ref{eq}) has been proposed by Tuval \cite{tuvel} to describe the motion of oxygen consumed by bacteria in a drop of water. Here $n$ and $c$ denote the density of Bacteria and concentration of oxygen, respectively. We also write the fluid velocity by $u$ and the associated pressure by $P$. In addition to random diffusion, the bacteria bias their movement to the favorable direction which is determined by the environment and distribution of oxygen consumed by the bacteria themselves. Meanwhile, both the oxygen and bacteria are supposed to be transported by the surrounding fluid. Let $\phi$ be a potential function; the fluid motion is described by incompressible Navier-Stokes equation and also influenced by external force $n\nabla\phi$, which can be understood as buoyant, electric or magnetic force of bacterial mass. This mechanism is an important variation of chemotaxis model, which has been extensively studied in the past 40 years; we refer to surveys \cite{hillen_painter,horstmann,bbtw} for a broad view.\\

In the paper \cite{tuvel}, $S=\chi\cdot\mathbb{I}$ with $\chi\in\mathbb{R}$, thus the cross diffusion term reduces to $\nabla\cdot(n\chi\nabla c)$, which indicates that the bacteria always move towards the higher concentration of oxygen. Therefore, a  coupled chemotaxis fluid model reads as
\begin{equation}\label{eq:system_without_rot}
\left\{
\begin{array}{rll}
 n_t&=&\Delta n - \nabla\cdot(\chi(c)n\nabla c) - u\cdot \nabla n,\\[4pt]
 c_t&=&\Delta c - nf(c) -u\cdot\nabla c, \\[4pt]
 u_t&=&\Delta u -(u\cdot\nabla)u+\nabla P+n\nabla \phi,\\[4pt]
\nabla \cdot u&=&0.
 \end{array}
\right.
\end{equation}
Actually, under suitable assumptions on $\chi$ and $f$, which are mild enough such that the prototypical choice $\chi(c)\equiv 1$ and $f(c)\equiv c$ is allowed, and a natural gradient-like functional for (\ref{eq:system_without_rot}) is expressed as
\begin{align}\label{ineq:functional}
\frac{d}{dt}\left(\intO n\ln n+\frac12\intO \frac{|\nabla c|^2}{c}\right)+\intO \left(\frac{|\nabla n|^2}{n}+c|D^2\ln c|^2\right)\le C\intO |u|^4
\end{align}
with some constant $C>0$. A crucial point to identify the above functional is that the term $\intO \chi(c)\nabla n\cdot\nabla c$ from a natural Lyapunov functional for the first equation can be cancelled by
a suitable testing procedure on the second equation. Starting from (\ref{ineq:functional}), a large number of articles have gained
 considerable results. Global classical solutions are demonstrated for two-dimensional bounded domains \cite{win_cpde}. Beyond this, a deeper understanding of the functional leads to boundedness of solutions for large initial data, and furthermore the solutions approach the spatially homogeneous equilibrium \cite{win_arma}:
\[
(n,c,u)\rightarrow(\nbar_0,0,0) \text{ as } t\to\infty,
 \]
 where $\nbar_0=\frac1{|\Om|}\intO n_0$. The convergence rates are studied later in \cite{zhang_li_dcdcb} for the convergent solutions. Concerning the case $N=3$, (\ref{ineq:functional}) is still crucial, \cite{win_cpde} asserts the existence of global weak solutions for the Stokes-governed system based on it. Recently, a global weak solution is constructed for the full Navier-Stokes system for large initial data \cite{win_hp}. More recently, with a concept of {\em eventual energy solution}, \cite{win_energ3d} shows that such solutions become smooth after finite time and uniformly converge to the constant steady state in the large time limit. For more results depending on the natural functional like (\ref{ineq:functional}), see, e.g., \cite{lorz,liu_lorz,duan_lorz_markowich,chae_kang_lee_smooth,tao_win_hp,win_hp,win_cpde} and the references therein.\\

However, in \cite{xue_othmer}, the authors suggest a wider choice of $S$ due to some complicated interaction neighborhood environment around cells. A kind of interactions between the cell motion speed and directional effects stemming from the action of gravity may result in abnormal mechanism -- they do not move directly to the direction of higher density of oxygen but with some rotation; so this requires $S$ to be a general matrix. Apart from the complexity as it stands, this tensor-valued chemotactic sensitivity also gives rise to some  difficulty in mathematical analysis.
Upon the aforementioned reasoning of (\ref{ineq:functional}), we may see that it heavily relies on the structure of the cross diffusion term. Here the term from the Lyapunov functional reads as $\intO \chi(c)nS\cdot\nabla c\cdot\nabla n$ and it is no longer cancellable for arbitrary choices of $S$. Thus when (\ref{ineq:functional}) is absent, it is much more difficult to study (\ref{eq}) from a mathematical point of view.
\\

Generally, stronger assumptions seem necessary for the existence of classical solutions.
For instance, considering the two-dimensional fluid-free system, that is, $u\equiv 0$, it is shown that the system admits global classical solutions which converge to the constant steady state if $\|c_0\|_{L^\infty(\Omega)}$ is sufficiently small \cite{li_xue_suen_win}. This is in sharp contrast to the case that $S$ is a scalar-valued sensitivity \cite{tao_win_eventual}, where (\ref{ineq:functional}) is still applicable, and finally it leads to boundedness of solutions and large time convergence without any smallness condition on the initial data. On the other hand, the same problem for large data are studied in \cite{win_2drot,win_2drotStokes} with or without fluid effect. In this case, it is shown that a certain generalized solution exists, and converge to the constant steady state in the large time limit. However, the results do not exclude singularity on intermediate time scales.\\

Considering porous medium type cell diffusion, that is, when the first equation in \eqref{eq} is replaced by $n_t=\Delta n^m-\nabla\cdot(n S \cdot\nabla c)$ with $m>1$, the existence of global weak solution is derived for any reasonable regular initial data, moreover, the solution is actually bounded \cite{cao_ishida}. Taking the fluid into account, the coupled Stokes and Navier-Stokes counterpart to the same problem are studied in \cite{tao_win_dcds} and \cite{I2014}, where the authors prove boundedness and global existence of weak solutions.\\

Due to the difficulty arising from the three-dimensional Navier-Stokes equation, only until very recently, the full chemotaxis-Navier-Stokes system (\ref{eq}) with scalar sensitivity is known to admit considerably weak solutions. Accordingly, many works have focused on the simplified Stokes coupled system and shown more progress. Assume that
\(
|S|\le C(1+n)^{-\alpha}\)
with $C>0$ and $\alpha>\frac16$, the global existence of locally bounded classical solution is constructed in \cite{cao_wang}.
 Considering porous medium variant of (\ref{eq}), that is, when the first equation becomes $n_t=\Delta n^m-\nabla\cdot(n\nabla c)-u\cdot\nabla n$, it is proved that locally bounded solutions exist and are locally bounded under the hypothesis that $m>\frac{8}{7}$ \cite{tao_win_hp}. If we assume in addition that $m>\frac 76$, upon a more robust approach, the solutions become globally bounded and converge to $(\overline{n_0},0,0)$ \cite{win_3d}.
 Without assuming superlinear diffusion, the question of boundedness of classical solution is actually more delicate and solved recently in \cite{cao_lan} that even matrix-valued sensitivity is allowed. It is shown that if \(\norm[L^p(\O)]{n-\bar{n}_0}\), \(\norm[L^q(\O)]{\nabla c_0}\) and \(\norm[L^N(\O)]{u_0}\) (with any $p>\frac N2$ and $q>N$) are small, the system admits a unique global classical solution which converge to the homogenous equilibrium. \\

The purpose of the present work is to study the full chemotaxis-Navier-Stokes system with tensor-valued sensitivity in dimension 2 and the corresponding chemotaxis-Stokes system in dimension 3. When the natural Lyapunov functional is lacking, we impose a smallness assumption on the initial data to get some uniform bound for the solution. Using this tool, we can prove global existence of classical solution and its large time behavior. Compared with \cite{cao_lan}, the smallness condition here is only on $\|c_0\|_{L^\infty(\Om)}$, meaning that small concentration of oxygen can force stability. This result coincides with the fluid-free system in \cite{li_xue_suen_win}. 
The convexity of the physical domain is unnecessary in this paper since we use a different approach from many previous works \cite{win_cpde}.\\

Before stating our main result, let us briefly introduce some elementary background of functional spaces, Stokes operator as well as their applications and some notations.

Let
$L_\sigma^p(\O)$  (\(1<p<\infty\)) denote solenoidal space equipped with $\norm[L^p(\O)]{\cdot}$ norm:
\[
L_\sigma^p(\O)£º=\{\varphi\in C_{0}^\infty(\O;\mathbb{R}^N)|\nabla\cdot\varphi=0\}
\]
The so-called Helmholtz-projection is defined as $\mathscr {P}: L^p(\O,R^N)\to L_\sigma^p(\O)$, which is a bounded operator. Let $A_p=-\mathscr P \Delta$
denote the Stokes operator
in $D(A_p)=L_\sigma^p(\O)\cap W^{2,p}(\O)\cap W_0^{1,p}(\O)$. From \cite{giga}, we know that $A$ is sectorial and generates analytical semigroup $(e^{-tA})_{t>0}$ in $L^p_{\sigma}(\O)$. We refer to \cite[Lemma 2.3]{cao_lan} for fundamental $L^p$-$L^q$ estimates for the semigroup. Moreover, since ${\rm Re}\, \sigma(A)>0$,
we can define $A^{-\alpha}$ with $\alpha>0$ and easily check that it is one-to-one.
Thus $A^\alpha$ is defined as the inverse of $A^{-\alpha}$ and $D(A^\alpha)=R(A^{-\alpha})$ \cite[Chapter 2.6]{pazy}. The following estimate is fundamental:
\begin{align}\label{fraction}
\|A^{\alpha}e^{tA}\|\le C_\alpha t^{-\alpha}e^{-\mu t} \text{ for } t>0 \text{ and for some } \mu>0.
\end{align}
Throughout the paper, we denote the first eigenvalue of $A$ by $\lambda_1'$, and by $\lambda_1$ the first nonzero eigenvalue of $-\Delta$ on $\Om$ under Neumann boundary conditions.
Moreover, we assume that
\begin{align}\label{1:SC2}
& s_{ij}\in C^2(\overline{\Omega}\times[0,\infty)\times [0,\infty)),\\
\label{1:S}
&|S(x,n,c)|:=\max_{i,j\in\{1,2\} }\{s_{ij}(x,n,c)\} \le S_0(c)\,\,
\mbox{for all }(x,n,c)\in\overline{\Omega}\times[0,\infty)\times [0,\infty),
\end{align}
where $S_0$ is a non-decreasing function on $[0,\infty)$.
The initial data are chosen as
\begin{align}\label{1:initial1}
\left\{
\begin{array}{llc}
  &n_0\in L^\infty(\O),\\[6pt]
  &c_0\in W^{1,q}(\O),\ q>N,\\[6pt]
  &u_0\in D(A^\alpha),\ \alpha\in (\tfrac{N}{4},1),
  \end{array}
  \right.
  \end{align}
and particularly
\begin{align}\label{1:initial2}
n_0\ge0,\;\;c_0\ge 0 \text{ on }\O.
\end{align}

Under the above assumptions and notations, our main result is as follows:
\begin{thm}\label{th1}
Let $N\in\{2,3\}$, $\O\subset\mathbb{R}^N$ be a bounded domain with smooth boundary.
Assume that $S$ fulfills (\ref{1:SC2}-\ref{1:S}). Either of the following conditions holds,
\begin{enumerate}
\item [(i)]  $N=2$, $\kappa=1$;
\item [(ii)] $N=3$, $\kappa=0$.
\end{enumerate}
There is $\delta_0>0$ with the following property: If the initial data fulfill (\ref{1:initial1}-\ref{1:initial2}), and
\bea{c0}
\|c_0\|_{\Lin}<\delta_0,
\eea
then (\ref{eq}) admits a global classical
solution $(n,c,u,P)$ which is bounded, and satisfies
\begin{equation}\label{sol_regularity}
\left\{
\begin{array}{llc}
&n\in C^{2,1}(\overline{\Om}\times(0,\infty))\cap C^0_{loc}(\overline{\O}\times(0,\infty)),\\[6pt]
&c\in C^{2,1}(\overline{\Om}\times(0,\infty))\cap
C^0_{loc}(\overline{\O}\times(0,\infty))\cap L^\infty((0,\infty);W^{1,q}(\O)),\\[6pt]
&u\in C^{2,1}(\overline{\Om}\times(0,\infty))\cap L^\infty((0,\infty);D(A^\alpha))\cap C^0_{loc}([0,\infty);L^2(\O)),\\[6pt]
&P\in L^1((0,\infty);W^{1,2}(\O)).
\end{array}
\right.
\end{equation}
\end{thm}
\begin{remark}
The uniqueness of classical solutions in the indicated class can be proved similarly as in \cite{win_cpde}.
\end{remark}
Apart from boundedness and global existence, we can also show each component converges to the homogenous equilibrium with {\em optimal} rates.
\begin{cor}\label{cor:con_regularity}
Under the assumptions of Theorem \ref{th1}, let $0<\alpha<\min\{\nbar_0,\lambda_1\}$ and $0<\alpha'<\min\{\alpha,\lambda'_1\}$. The solution of \eqref{eq} has the property that there is $C>0$ fulfilling
\[
\norm[L^\infty(\O)]{n(\cdot,t)-\nbar_0} \leq Ce^{-\alpha t}, \qquad \norm[W^{1,q_0}(\O)]{ c(\cdot,t)}\leq Ce^{-\alpha t},\qquad \norm[L^\infty(\O)]{u(\cdot,t)}\leq Ce^{-\alpha' t} \mbox{ for all } t>0.
\]
\end{cor}

We note that compared with the result in \cite{win_arma}, Theorem \ref{th1} furthermore has restrictions on the size of initial data (\ref{c0}), which seems necessary for the existence of classical solutions. As a subcase of (\ref{eq}), results on the corresponding fluid-free version are not yet rich: Without assuming small data, the global generalized solutions constructed in \cite{win_2drot} still possibly become unbounded in the intermediate time; Only additionally assuming $\|c_0\|_{\Lin}$ small, global classical solutions are known to exist and blow-up is entirely ruled out \cite{li_xue_suen_win}. When the system is coupled with fluid component, our results give the same condition which guarantee the global existence of smooth solution.\\

The plan of the paper is as follows:

In Section 2, we approximate the problem by a well-posed system (see (\ref{ep}) later). Section 3-5 are devoted to study the boundedness of regularized problem, we will see the bounds are independent of the way we regularize the problem. Thus upon appropriate estimates, we can let $\ep\to 0$ to obtain limit functions of the regularized solutions. This procedure is done in Section 6, and also these limit functions are shown to be smooth enough and solves (\ref{eq}) classically for any positive time. In Section 7, we prove stabilization of the solution by applying the result from \cite{cao_lan}.

\section{ Approximation}
Since it is convenient to deal with the Neumann boundary conditions for both $n$ and $c$, we follow the same approximation procedure as in \cite{li_xue_suen_win}. Let $\ep\in(0,1)$, we find a family of functions $\{\rho_\ep\}_{\ep\in(0,1)}$ satisfying
\begin{align}\label{2:roh_ep}
\rho_\ep\in C_0^\infty(\Om)\;\; \text{ with }\;\;0\le\rho_\ep\le1 \text{ in } \Om \;\;\text{and} \;\;\rho_\ep\nearrow 1\text{ in } \Om \text{ as } \ep\searrow 0,
\end{align}
and define
\begin{align}\label{2:Sep}
S_\ep(x,\nep,\cep)=\rho_\ep(x)S(x,n,c),\;\; x\in\bar{\Om}.
\end{align}
Then we have $S_\ep(x,n,c)=0$ on $\pa\O$ and
\begin{align}\label{2:Sepnorm}
|S_\ep(x,\nep,\cep)|\le S_0(\norm[L^\infty(\Om)]{c_0})\;\; \text{ for all } x\in\Om,\,\nep>0,\,\cep>0.
\end{align}
Now we consider the following regularized problem
\begin{equation}\label{ep}
\left\{
\begin{array}{llc}
n_{\ep t}=\Delta \nep-\nabla\cdot(\nep S_\ep(x,\nep,\cep)\cdot\nabla \cep)+\uep\cdot\nabla \nep,
&(x,t)\in \Omega\times (0,T),\\[4pt]
\displaystyle
c_{\ep t}=\Delta \cep-\nep\cep+\uep\cdot\nabla \cep , &(x,t)\in\Omega\times (0,T),\\[4pt]
\displaystyle
u_{\ep t}=\Delta \uep-\kappa(\uep\cdot\nabla)\uep+\nabla P_\ep+\nep\nabla\phi , \,\nabla\cdot\uep=0, &(x,t)\in\Omega\times (0,T),\\[4pt]
\displaystyle
\nabla \nep\cdot\nu=\nabla \cep\cdot\nu=0,\;\uep=0, &(x,t)\in\pa\O\times (0,T),\\[4pt]
\displaystyle
\nep(x,0)=n_0(x), \cep(x,0)=c_0(x), \uep(x,0)=u_0(x),
&x\in\O.
\end{array}
\right.
\end{equation}

Without essential difficulty, the above system is locally solvable in the classical sense by an adaption of well-established fixed point argument \cite[Lemma 2.1]{win_cpde}.
We give the following lemma without proof.

\begin{lem}\label{lem:localexistence}
Let $N\in\{2,3\}$, $\O\subset\mathbb{R}^N$ be a bounded domain with smooth boundary, and $\kappa\in \mathbb{R}$.
Assume initial data $(n_0,c_0,u_0)$ satisfy (\ref{1:initial1}) and (\ref{1:initial2}),
and $S$ fulfills (\ref{1:SC2}-\ref{1:S}).
Then there exist $T_{\max}\in(0,\infty]$ and
a unique classical solution $(\nep,\cep,\uep,P_\ep)$ to \eqref{ep} in $\O\times [0,T_{\max})$ with $\nep,\cep>0$.
Moreover, if $T_{\max}<\infty$, then
$$
  \|\nep(\cdot,t)\|_{L^\infty(\O)}+\|\cep(\cdot,t)\|_{W^{1,q}(\O)}+\|A^\alpha \uep(\cdot,t)\|_{L^2(\O)}\to\infty
  \quad \mbox{as } t\nearrow T_{max}.
$$
\end{lem}

In order to see the global existence and qualitative behavior of the regularized problem, it is sufficient to show boundedness for each criterion in the above lemma. The following lemma is immediately obtained upon observation.

\begin{lem}\label{lem:nc}
  Let $(\nep,\cep,\uep,P_\ep)$ be a classical solution of \eqref{ep}. It follows that
  \begin{align}\label{2:mass}
  &\|\nep(\cdot,t)\|_{L^1(\O)}=\|n_0 \|_{L^1(\O)} ,\\
  \label{2:c-Lp}
  &\|\cep(\cdot,t)\|_{L^\infty(\O)}\le\|c_0\|_{L^\infty(\O)} \ \text { for all}\ t\in(0,T_{max}).
  \end{align}
\end{lem}
\begin{proof}
The mass conservation \eqref{2:mass} is obtained by
integrating the first equation of \eqref{eq} on $\O$ and using the Neumann boundary condition.
Since $\nep$ and $\cep$ are nonnegative,
an application of the maximum principle to the second equation yields \eqref{2:c-Lp}.
\end{proof}
We then obtain boundedness and global existence for the regularized problem (\ref{ep}).
\begin{proposition}\label{epboundedness}
Let $N\in\{2,3\}$,  $\O\subset\mathbb{R}^N$ be a bounded domain with smooth boundary.
Assume that $S$ fulfills (\ref{1:SC2}-\ref{1:S}). Either of the following conditions holds
 \begin{enumerate}
 \item [(i)] $N=2$, $\kappa=1$;
 \item [(ii)]  $N=3$, $\kappa=0$.
 \end{enumerate}
Then there exists $\delta_0>0$ with the following property:
If the initial data fulfill (\ref{1:initial1}-\ref{1:initial2}), and
\bea{c0}
\|c_0\|_{\Lin}<\delta_0,
\eea
then (\ref{ep}) admits a global classical solution $(\nep,\cep,\uep)$. And there is $C>0$ such that
\begin{align}
&\norm[\Lin]{\nep(\cdot,t)}\le C,\;\;\norm[W^{1,q}(\O)]{\cep(\cdot,t)}\le C, \;\;\norm[L^2(\O)]{A^\alpha\uep(\cdot,t)}\le C
\end{align}
for all $t\in(0,\infty)$ and all $\ep\in(0,1)$.
\end{proposition}
We will prove boundedness for the 2-dimensional and 3-dimensional cases in Section 4 and Section 5, respectively. However, the $L^p(\O)$ estimate for $\nep$ derived in the next section will be applied to both.

\section{A priori estimate for $\nep$}
In this section, we obtain boundedness of $\nep$ in $L^p(\Om)$ under the assumption that $\norm[L^\infty(\Om)]{c_0}$ is suitably small. The approach is based on the weighted estimate of $\intO \nep^p\varphi(\cep)$ with appropriate choice of $\varphi$ which has been developed in \cite{win_2012mathnachr} and adapted to the consumed type signal in \cite{tao_jmaa,win_arma}.

\begin{lem}\label{lem:Lp}
Let $p>1$, there is $\delta_0:=\delta_0(p)>0$ and $C>0$ have the property: If the initial data satisfy (\ref{1:initial1}-\ref{1:initial2}) and
\begin{align}\label{3:c0}
\|c_0\|_{L^\infty(\Om)}< \delta_0,
\end{align}
then for all $\ep\in(0,1)$, we have
\begin{align}\label{3:1}
&\norm[L^p(\Om)]{\nep(\cdot,t)}\le C\,\, \text{ for all } t\in(0,T_{\max}),\\
\label{3:2}
\text{ and } &\int_0^{T_{\max}}\into \nep^{p-2}|\nabla\nep|^2\le C.
\end{align}
\end{lem}

\begin{remark}
The argument does not depend on dimension $N$ or the value of $\kappa$.
\end{remark}
\begin{proof}
Let $p>1$, $0<h<\frac1{48}$. We can find
$\delta_0$ satisfying
\begin{align}\label{3:del1}
& 3p(p-1)\delta_0^2 S_0^2(\delta_0)\le h(h+1),\\\label{3:del2}
& 3p \delta_0 S_0(\delta_0)\le h+1,
\end{align}
where $S_0$ is non-decreasing function as introduced in \eqref{1:S}. Under the assumption of (\ref{3:c0}), we can define $\varphi(\cep)=(\delta_0-\cep)^{-h}$ according to \eqref{2:c-Lp}, thus $\varphi(\cep)>0$. Elementary calculus shows that
\begin{align}\label{phi'}
&\varphi'(\cep)=h(\delta_0-\cep)^{-h-1}>0,\\\label{phi''}
&\varphi''(\cep)= h(h+1)(\delta_0-\cep)^{-h-2}>0.
\end{align}
Using the first two equations in (\ref{ep}), upon integrating by part we obtain
\begin{align}\nn
&\frac{d}{dt}\intO \nep^p\varphi(\cep)\\\nn
&=\intO p\nep^{p-1}\varphi(\cep)(\Delta \nep-\nabla\cdot(\nep S_\ep\cdot\nabla \cep)-\uep\cdot\nabla \nep)+\intO \nep^p\varphi'(\cep)(\Delta \cep-\nep\cep-\uep\cdot\nabla \cep)\\\nn
&=-\intO\nabla\nep\cdot(p(p-1)\nep^{p-2}\varphi(\cep)\nabla \nep+p\nep^{p-1}\varphi'(\cep)\nabla \cep)\\\nn
&~~~~~+\intO \nep S_\ep(x,\nep,\cep)\cdot\nabla \cep\cdot\left(p(p-1)\varphi(\cep)\nep^{p-2}\nabla\nep+p\nep^{p-1}\varphi'(\cep)\nabla\cep\right)\\\nn
&~~~~~~-\intO p\nep^{p-1} \varphi(\cep)\uep\cdot\nabla\nep-\intO \nabla\cep\cdot(p\nep^{p-1}\varphi'(\cep)\nabla\nep+\nep^p\varphi''(\cep)\nabla\cep)\\\nn
&~~~~~~~~-\intO \nep^p\varphi'(\cep)\uep\cdot\nabla\cep-\intO \nep^{p+1}\cep\varphi'(\cep)\\\nn
&=-p(p-1)\into \nep^{p-2}\varphi(\cep)|\nabla\nep|^2-p\into\nep^{p-1}\varphi'(\cep)\nabla\nep\cdot\nabla\cep\\\nn
&~~~~
+p(p-1)\into\nep^{p-1}\varphi(\cep)S_\ep(x,\nep,\cep)\cdot\nabla\cep\cdot\nabla\nep
+p\into\nep^p\varphi'(\cep)S_\ep(x,\nep,\cep)\cdot\nabla\cep\cdot\nabla\cep
\\\label{3:1.1}
&~~~~~-p\into\nep^{p-1}\varphi'(\cep)\nabla\nep\cdot\nabla\cep-\into\nep^p\varphi''(\cep)|\nabla\cep|^2
-\into\nep^{p+1}\varphi'(\cep)c
\end{align}
for all $t\in(0,T_{\max})$,
where we have used the identity
\begin{align}\nn
-p\intO\nep^{p-1}\varphi(\cep)\uep\cdot\nabla\nep-\intO\nep^p\varphi'(\cep)\uep\cdot\nabla\cep
&=-\intO\varphi(\cep)\uep\cdot\nabla \nep^p-\intO\nep^p\uep\cdot\nabla\varphi(\cep)\\\nn
&=\intO \nep^p\varphi(\cep)(\nabla\cdot\uep)=0.
\end{align}
In light of (\ref{2:Sepnorm}), we find that
\begin{align}\nn
&\frac{d}{dt}\intO \nep^p\varphi(\cep)+p(p-1)\intO\varphi(\cep)\nep^{p-2}|\nabla\nep|^2
+\intO\nep^p\varphi''(\cep)|\nabla\cep|^2\\\nn
&=p(p-1)S_0(\norm[\Lin]{c_0})\intO\nep^{p-1}\varphi(\cep)|\nabla\nep||\nabla\cep|+2p\intO\nep^{p-1} \varphi'(\cep)|\nabla\nep||\nabla\cep|\\\label{3:1.2}
&~~~~~~~~~~~~~~~~~~
+pS_0(\norm[\Lin]{c_0})\intO\nep^p\varphi'(\cep)|\nabla\cep|^2
\end{align}
for all $t\in(0,T_{\max})$.
Here Young's inequality yields that
\begin{align}\nn
p(p-1)S_0(\norm[\Lin]{c_0})\intO \nep^{p-1}\varphi(\cep)|\nabla\nep||\nabla\cep|&\le\frac{p(p-1)}{4} \intO\nep^{p-2}\varphi(\cep)|\nabla\nep|^2\\\label{3:1.4}
&~~~~~~~+p(p-1)S_0^2(\norm[\Lin]{c_0})\intO\nep^p\varphi(\cep)|\nabla\cep|^2,\\
\label{3:1.3}
2p\intO \nep^{p-1}\varphi'(\cep)|\nabla\nep||\nabla\cep|\le\frac{p(p-1)}{4}&\intO \nep^{p-2}\varphi(\cep)|\nabla \nep|^2+16\intO \nep^p\frac{\varphi'^2(\cep)}{\varphi(\cep)}|\nabla\cep|^2,
\end{align}
We see that (\ref{3:1.2}-\ref{3:1.4}) imply
\begin{align}\nn
&\frac{d}{dt}\intO \nep^p\varphi(\cep)+\frac{p(p-1)}{2}\intO\nep^{p-2}\varphi(\cep)|\nabla\nep|^2
\\\label{3:1.5}
&+\into\nep^p|\nabla\cep|^2\left(\varphi''(\cep)-16\frac{\varphi'^2(\cep)}{\varphi(\cep)}
-p(p-1)S_0^2(\norm[\Lin]{c_0})\varphi(\cep)-pS_0(\norm[\Lin]{c_0})\varphi'(\cep)\right)\le 0
\end{align}
for all $t\in(0,T_{\max})$. Now using (\ref{3:del1}-\ref{3:del2}), and in view of the fact that $S_0(\delta)$ is non-decreasing, we see that
\begin{align}\nn
&16\frac{\varphi'^2(\cep)}{\varphi(\cep)}=16h^2(\delta_0-\cep)^{-h-2}\le \frac13\varphi''(\cep),\\\nn
&p(p-1)S_0^2(\delta_0)\varphi(\cep)=p(p-1)S_0^2(\delta_0)(\delta_0-\cep)^{-h}\le \frac13\varphi''(\cep),\\\nn
&pS_0(\delta_0)\varphi'(\cep)=hpS_0(\delta_0)(\delta_0-\cep)^{-h-1}\le \frac13\varphi''(\cep).
\end{align}
Thus the term $\displaystyle\into\nep^p|\nabla\cep|^2\left(\varphi''(\cep)-16\frac{\varphi'^2(\cep)}{\varphi(\cep)}
-p(p-1)S_0^2(\delta_0)\varphi(\cep)-pS_0(\delta_0)\varphi'(\cep)\right)$ on the right hand side of \eqref{3:1.5} is nonnegative, we immediately deduce that
\begin{align}
\frac{d}{dt}\into\nep^p\varphi(\cep)+\frac{p(p-1)}{2}\into \nep^{p-2}\varphi(\cep)|\nabla\nep|^2\le0, \text{ for all } t\in(0,T_{\max}).
\end{align}
Since $\varphi(\cep)$ is bounded from above and below, (\ref{3:1}) and (\ref{3:2}) result from the above inequality upon integrating on $(0,T_{\max})$.
\end{proof}

\section{ Boundedness in two-dimensional case ($N=2$, $\kappa=1$)}

We expect that the $L^p(\Om)$ estimate obtained in the last section guarantees boundedness of $\nep$ in $L^\infty(\Om)$ as in the fluid-free system. However, the iteration procedure is much more delicate due to the appearance of the transport terms in the current case. Since the regularity of $\nabla\cep$ is crucial, which is also associated to the regularity of $\uep$, we will first get the suitable regularity of $\uep$. More precisely, the $L^2(\Om)$ norm of $\nabla\uep$ implies boundedness of $\norm[L^p(\Om)]{u(\cdot,t)}$ for any $p>1$. This is sufficient to prove boundedness of $\norm[L^\infty(\Om)]{\nabla\cep(\cdot,t)}$. 

\subsection{Boundedness of $\|\nabla \uep(\cdot,t)\|_{L^2(\Om)}$}

\begin{lem}\label{lem:4.1}
Let $N\in\{2,3\}$.
Suppose that
\begin{align}\label{4:1}\mathop{\sup}\limits_{t\in(0,T_{\max})}
\norm[L^2(\O)]{\nep(\cdot,t)}<\infty.\end{align}
Then there exits $C>0$ such that for any $\ep>0$
\begin{align}\label{2.1.0}
&\|\uep(\cdot,t)\|_{L^2(\Om)}<C\;\;\text{ for all } t\in(0,T_{\max}),\\\label{2.1.0'}
&\int_k^{\min\{k+1,T_{\max}\}}\into |\nabla \uep|^2<C\;\;\text{ for all } k\in\mathbb{T}:=\{ s\in\mathbb{N} , s\le [T_{\max}]\}.
\end{align}
\end{lem}
\begin{proof}
Testing the third equation with $\uep$, integrating by parts and Young's inequality yield that
\begin{align}\nn
\frac12\frac{d}{dt}\into |\uep|^2+\into|\nabla \uep|^2&=\into \nep\nabla\phi\cdot \uep\\\label{2.1.1}
&\le \frac{\lambda_1'}{2}\into|\uep|^2
+\frac{1}{2\lambda_1'}\|\nabla\phi\|_{L^\infty(\Om)}^2\into \nep^2\end{align}
for all $t\in(0,T_{\max})$.
The Poincar\'e inequality combined with (\ref{4:1}) implies the existence of $c_1>0$ such that
\begin{align}
\frac{d}{dt}\into |\uep|^2+\lambda_1'\into |\uep|^2\le c_1
\end{align}
for all $t\in(0,T_{\max})$. Thus (\ref{2.1.0}) is obtained by the comparison theorem. Now we integrate (\ref{2.1.1})
on $(k,k+1)$ ($k\in\mathbb{T}$) to find that
(\ref{2.1.0'}) holds due to (\ref{2.1.0}).
\end{proof}
\begin{remark}
We note that the lemma does not depends on the dimensions, thus we are able to use the same reasoning in other situations, e.g. Lemma \ref{lem:con:u}.
\end{remark}

Base on (4.17) in \cite{win_cpde}, we can prove $\|\nabla \uep(\cdot,t)\|_{L^2(\Om)}$ is bounded with the aid of (\ref{2.1.0'}). The assumption $N=2$ is crucial
here.
\begin{lem}\label{lem:nablau}
Let $N=2$. Suppose that
\begin{align}
\mathop{\sup}\limits_{t\in(0,T_{\max})}
\norm[L^2(\O)]{\nep(\cdot,t)}<\infty.
\end{align}
There is $C>0$ fulfilling for any $\ep>0$
\begin{align}
\|\nabla \uep(\cdot,t)\|_{L^2(\Om)}\le C\,\,\text{ for all } t\in(0,T_{\max}).
\end{align}
\end{lem}

\begin{proof}
First we apply Lemma \ref{lem:4.1} to obtain (\ref{2.1.0}) and (\ref{2.1.0'}).
Let $A=-\mathscr{P}\Delta$ and hence $\|A^\frac12 \uep\|_{L^2(\Om)}=\|\nabla \uep\|_{L^2(\Om)}$. Testing the third equation by $A\uep$ implies
\begin{align}\nn
\frac12\frac{d}{dt}\into|A^\frac12 \uep|^2+\into |A\uep|^2&=\into A\uep(\uep\cdot\nabla)\uep+\into \nep\nabla\phi A\uep\\\nn
&\le \frac14\into |A\uep|^2+\into |\uep|^2|\nabla \uep|^2+\frac14\into|A\uep|^2+\norm[\Lin]{\nabla\phi}^2\into \nep^2\\\label{2.2.11}
&\le \into |\uep|^2|\nabla \uep|^2+\frac12\into |A\uep|^2+\norm[\Lin]{\nabla\phi}^2\into \nep^2
\end{align}
for all $t\in(0,T_{\max})$. By Young's inequality, an interpolation inequality for $\norm[L^4(\O)]{\uep}$ and
$\norm[L^4(\Om)]{\nabla \uep}$ ({ see also in \cite[Proof of Theorem 1.1]{win_cpde} }), and the equivalence between the norms $\norm[L^2(\O)]{A(\cdot)}$ and $\norm[W^{2,2}(\O)]{\cdot}$
\begin{align}\nn
\into|\uep|^2|\nabla \uep|^2
&\le (\into |\uep|^4)^\frac12(\into|\nabla \uep|^4)^\frac12\\\nn
&\le (\into|\nabla \uep|^2)^\frac 12(\into |\uep|^2)^\frac12(\into|A\uep|^2)^\frac12(\into |\nabla \uep|^2)^\frac12\\\label{2.2.12}
&\le \frac12\into|A\uep|^2 +\frac{1}2(\into|\uep|^2)(\into|\nabla \uep|^2)^2.
\end{align}
We see that (\ref{2.2.11}) and (\ref{2.2.12}) in conjunction with our assumption and (\ref{2.1.0}) imply that there is $c_1>0$ fulfilling
\begin{align}
\frac{d}{dt}\into|\nabla \uep|^2+\into |A\uep|^2\le c_1\left(\into |\nabla \uep|^2+1\right)^2
\end{align}
for all $t\in(0,T_{\max})$.
Let $y(t):=\into|\nabla \uep(\cdot,t)|^2+1$, thus $y(t)$ satisfies
\begin{align}\label{2.2.2}
y'(t)\le c_1y^2(t)
\end{align}
for all $t\in[k,\min\{k+1,T_{\max}\})$.

If $T_{\max}>1$, for all $k\in\mathbb{T}$, Lemma \ref{lem:4.1} warrants the existences of $c_2>0$ and $s_k\in[k,k+1]$ such that
\begin{align}\label{2.2.3}
y(s_k)\le c_2,\,\,\int_k^{k+1}y(s)ds\le c_2.
\end{align}
We deduce from (\ref{2.2.2}-\ref{2.2.3}) that
\begin{align}\label{2.2.1}
y(t)\le e^{c_1\int_{s_k}^ty(s)ds}y(s_k)\le e^{c_1\int_{k}^{\min\{k+2,T_{\max}\}}y(s)ds}y(s_k)\le e^{2c_1c_2}c_2
\end{align}
for all $t\in[k+1,\min\{k+2,T_{\max}\}]\subset [s_k,\min\{k+2,T_{\max}\})$ ($k\in\mathbb{T}$). Thus (\ref{2.2.1}) holds for all $t\in[1,T_{\max})$.
A similar reasoning gives
\begin{align}y(t)\le e^{c_1\int_0^1y(s)ds}y(0)\le e^{c_1c_2}y(0) \mbox{ for all }t\in[0,1].
\end{align}
If $T_{\max}<1$, it is easy to see the above estimate still holds for $t\in[0,T_{\max})$.
Thus the proof is complete by letting $C:=\max\{e^{2c_1c_2}c_2,e^{c_1c_2}\|\nabla u_0\|_{L^2(\Om)}\}$.
\end{proof}

The following lemma is an immediate consequence from Sobolev embedding theorem for dimension 2.
\begin{lem}
Let $N=2$. Suppose that
\begin{align}
\mathop{\sup}\limits_{t\in(0,T_{\max})}
\norm[L^2(\O)]{\nep(\cdot,t)}<\infty.
\end{align}
Then for any $1<p<\infty$, there is $C>0$ such that for any $\ep>0$
\begin{align}\label{4:3}
\|\uep(\cdot,t)\|_{L^p(\Om)}\le C\,\,\text{ for all }t\in(0,T_{\max}).
\end{align}
\end{lem}

\subsection{Boundedness of $\|\nabla \cep(\cdot,t)\|_{L^\infty(\Om)}$}
%


Now we are in a position to get higher regularity of $\nabla\cep$, the approach is carried out by fixed-point argument involving $L^p$-$L^q$ estimates for semigroups combined with a typical integral estimate, which is borrowed from from \cite{W4,cao_lan}.
\begin{lem}\label{lem:integralestimates}
For all $\eta>0$ there is $C=C(\eta)>0$ such that for all $\alpha,\beta\in [0,1-\eta)$, and $\gamma,\delta\in \mathbb{R}$ satisfying $\frac1\eta\ge\gamma-\delta\ge\eta$, we have
\[
 \int_0^t (1+s^{-\alpha})(1+(t-s)^{-\beta}) e^{-\gamma s} e^{- \delta(t-s)} ds \leq C(\eta) e^{-\min\{\gamma,\delta\} t} (1+t^{\min\set{0, 1-\alpha-\beta}})\;\;\mbox{ for all }t>0.
\]
\end{lem}

\begin{lem}\label{lem:ns_nablac}
Let $N=2$, $p_0>2$. Suppose that
\[\mathop{\sup}\limits_{t\in(0,T_{\max})}\norm[L^{p_0}(\O)]{\nep(\cdot,t)}
<\infty.\]
Then there is $C>0$ such that for any $\ep>0$
\begin{align}\label{4:5}
\|\nabla \cep(\cdot,t)\|_{L^\infty(\Om)}\le C\;\;\text{ for all } t\in(0,T_{\max}).
\end{align}
\end{lem}

\begin{proof}
The variation of constants formula associated to $\cep$ implies
\begin{align}\nn
\|\nabla \cep(\cdot,t)\|_{L^\infty(\Om)}&\le \|\nabla e^{t\Delta}c_0\|_{L^\infty(\Om)}+\int_0^t\|\nabla e^{(t-s)\Delta}\nep(\cdot,s)\cep(\cdot,s)\|_{L^\infty(\Om)}ds\\\label{5.1.1}
&~~~~~~+\int_0^t\|\nabla e^{(t-s)\Delta}(\uep(\cdot,s)\cdot\nabla \cep(\cdot,s))\|_{L^\infty(\Om)}ds.
\end{align}
for all $t\in(0,T_{\max})$.
Recall that by the classical $L^p$-$L^q$ estimates for Neumann semigroup, there is $c_1>0$ such that
\begin{align}\label{5.1.2}
\|\nabla e^{t\Delta}c_0\|_{L^\infty(\Om)}\le c_1\|\nabla c_0\|_{L^\infty(\Om)}
\end{align}
for all $t\in(0,T_{\max})$ and for all $c_0\in W^{1,\infty}(\Om)$.
Since $p_0>2$, $L^p$-$L^q$ estimates yields that
\begin{align}\nn&
\int_0^t\|\nabla e^{(t-s)\Delta}\nep(\cdot,s)\cep(\cdot,s)\|_{L^\infty(\Om)}ds\\\nn
&\le \int_0^tc_1(1+(t-s)^{-\frac12-\frac1{p_0}})e^{-\lambda_1(t-s)}\|\nep(\cdot,s)\cep(\cdot,s)\|_{L^{p_0}(\Om)}ds
\\\label{5.1.3}
&\le \int_0^t c_1(1+(t-s)^{-\frac12-\frac1{p_0}})e^{-\lambda_1(t-s)}\|\nep(\cdot,s)\|_{L^{p_0}(\Om)}
\|\cep(\cdot,s)\|_{L^\infty(\Om)}ds,
\end{align}
for all $t\in(0,T_{\max})$,
which is bounded by (\ref{4:1}) and (\ref{2:c-Lp}).
Next we fix $p>2$ and moreover $p_1,p_2\in(p,\infty)$ satisfying $\frac1p=\frac{1}{p_1}+\frac{1}{p_2}$. Let $\theta=1-\frac{2}{p_2}\in(0,1)$, we thereby obtain
\begin{align}\nn
&~~~~\int_0^t\|\nabla e^{(t-s)\Delta}(\uep(\cdot,t)\cdot\nabla \cep(\cdot,t))\|_{L^\infty(\Om)}ds\\\nn
&\le \int_0^t c_1(1+(t-s)^{-\frac12-\frac1p})e^{-\lambda_1(t-s)}\|\uep(\cdot,t)\cdot\nabla \cep(\cdot,t)\|_{L^p(\Om)}ds\\\nn
&\le \int_0^t c_1(1+(t-s)^{-\frac12-\frac1p})e^{-\lambda_1(t-s)}\|\uep(\cdot,t)\|_{L^{p_1}(\Om)}\|\nabla \cep(\cdot,t)\|_{L^{p_2}(\Om)}ds\\\nn
&\le \int_0^t c_1(1+(t-s)^{-\frac12-\frac1p})e^{-\lambda_1(t-s)}\|\uep(\cdot,t)\|_{L^{p_1}(\Om)}
(\|\nabla \cep(\cdot,t)\|_{\Li}^\theta
\|\cep(\cdot,t)\|_{L^\infty(\Om)}^{1-\theta}\\\label{5.1.4}
&~~~~~~~~~~~~~~~~~~~~~~~~~~~~~~~~~~~~~~~~~~~~~~~~~~~~~~~~~~~+\|\cep(\cdot,s)\|_{\Li})ds
\end{align}
for all $t\in(0,T_{\max})$.
Let $T\in(0,T_{\max})$, and $M:=\mathop{\sup}\limits_{t\in(0,T)} \|\nabla \cep(\cdot,t)\|_{L^\infty(\Om)}$.
We see from (\ref{5.1.1}-\ref{5.1.4}) that
\begin{align*}
M\le c_2+c_2M^{\theta},
\end{align*}
with some $c_2>0$.
Since $\theta<1$, (\ref{4:5}) is obtained by Young's inequality.
\end{proof}

\subsection{Boundedness of $\nep$}
\begin{lem}\label{lem:4.6}
Let $N=2$, $p_0>2$. Suppose that
 \[\mathop{\sup}\limits_{t\in(0,T_{\max})}\norm[L^{p_0}(\O)]{\nep(\cdot,t)}
 <\infty.\]
Then there is $C>0$ such that for any $\ep>0$
\begin{align}
\norm[L^\infty(\Om)]{\nep(\cdot,t)}\le C \;\; \text{ for all } t\in(0,T_{\max}).
\end{align}
\end{lem}
\begin{proof}
Following the variation-of-constants formula, we see that
\begin{align}\nn
\|\nep(\cdot,t)\|_{L^\infty(\Om)}&\le\|e^{t\Delta}n_0\|_{L^\infty(\Om)}
+\int_0^t\|e^{(t-s)\Delta}\nabla\cdot(\nep S_\ep(\cdot,\nep,\cep)\cdot\nabla\cep)(\cdot,s)\|_{L^\infty(\Om)}ds\\
&~~~~~~~~~~~~~~+\int_0^t\|e^{(t-s)\Delta}\uep(\cdot,s)\cdot\nabla\nep(\cdot,s)\|_{L^\infty(\Om)}ds
\end{align}
for all $t\in(0,T_{\max})$. The first term can be estimated as
\begin{align}
\|e^{t\Delta}n_0\|_{L^\infty(\Om)}\le c_1\|n_0\|_{L^\infty(\Om)}\;\; \text{ for all } t\in(0,T_{\max})
\end{align}
with some $c_1>0$.
Moreover, applying $L^p$-$L^q$  for Neumann semigroup, we obtain
$c_2>0$ such that
\begin{align}\nn
&\int_0^t\|e^{(t-s)\Delta}\nabla\cdot(\nep S_\ep(\cdot,\nep,\cep)\cdot\nabla\cep)(\cdot,s)\|_{L^\infty(\Om)}ds\\\nn
&\le c_2\int_0^t(1+(t-s)^{-\frac12-\frac1{p_0}})e^{-\lambda_1(t-s)}
\|(\nep S_\ep(\cdot,\nep,\cep)\cdot\nabla\cep)(\cdot,s)\|_{L^{p_0}(\Om)}ds\\\label{4:3.1}
&\le c_2S_0(\norm[\Lin]{c_0})\int_0^t(1+(t-s)^{-\frac12-\frac1{p_0}})e^{-\lambda_1(t-s)}\|\nep(\cdot,s)\|_{L^{p_0}(\Om)}
\|\nabla\cep(\cdot,s)\|_{L^\infty(\Om)}ds
\end{align}
for all $t\in(0,T_{\max})$. By (\ref{4:1}) and (\ref{4:5}), we know the right hand side of (\ref{4:3.1}) is bounded.
Noting that $\uep\cdot\nabla\nep=\nabla\cdot(\nep\uep)$, we pick $p>2$ and $p'>p$ such that $\frac1p=\frac1{p_0}+\frac1{p'}$, a similar reasoning as the above inequality shows that
\begin{align}\nn
&~~\int_0^t\|e^{(t-s)\Delta}\uep(\cdot,s)\cdot\nabla\nep(\cdot,s)\|_{L^\infty(\Om)}ds\\\nn
&=\int_0^t\|e^{(t-s)\Delta}\nabla\cdot(\nep(\cdot,s)\uep(\cdot,s))\|_{L^\infty(\Om)}ds\\\nn
&\le c_2\int_0^t(1+(t-s)^{-\frac12-\frac1{p}})e^{-\lambda_1(t-s)}\|\nep(\cdot,s)\uep(\cdot,s)\|_{L^p(\Om)}ds\\\nn
&\le
c_2\int_0^t(1+(t-s)^{-\frac12-\frac1{p}})e^{-\lambda_1(t-s)}
\|\nep(\cdot,s)\|_{L^{p_0}(\Om)}\|u(\cdot,s)\|_{L^{p'}(\Om)}ds
\end{align}
for all $t\in(0,T_{\max})$ due to (\ref{4:1}) and (\ref{4:3}), it is bounded by Lemma \ref{lem:integralestimates}. Thus we complete the proof by collecting the above estimates.
\end{proof}
\subsection{Proof of (i) in Proposition \ref{epboundedness}}

In order to prove global existence of the solution, it is left to show boundedness of $\norm[L^2(\O)]{A^\alpha \uep(\cdot,t)}$  due to the extensive criterion.

\begin{lem}\label{lem:boundu}
Suppose that
\[
\mathop{\sup}\limits_{t\in(0,T_{\max})}\norm[L^2(\O)]{\nep(\cdot,t)}<\infty,
\;\mathop{\sup}\limits_{t\in(0,T_{\max})}\norm[L^2(\O)]{\uep(\cdot,t)}<\infty,
\;\mathop{\sup}\limits_{t\in(0,T_{\max})}\norm[L^2(\O)]{\nabla \uep(\cdot,t)}<\infty.
\]
Then there is $C>0$ such that for every $\ep>0$
\begin{align}\label{lem:uinfty}
\|A^\alpha \uep(\cdot,t)\|_{L^2(\Om)}\le C\;\;\text{ for all }t\in(0,T_{\max}).
\end{align}
\end{lem}
\begin{proof}
Let $T>0$, we first define $M(t):=\|A^\alpha\uep(\cdot,t)\|_{L^2(\O)}$ for $t\in(0,T)$. Let $a=\frac{N}{4\alpha}$, from the Gagliardo-Nirenberg inequality and \cite[Lemma 2.3(iv)]{cao_lan} we know that there is constant $c_1>0$ such that
\begin{align}\label{interpolation}
\|\uep\|_{\Li}\le c_1\|A^\alpha\uep\|_{L^2(\Om)}^a\|\uep\|_{L^2(\O)}^{1-a}.
\end{align}
We apply $A^\alpha$ to both sides of the third equation in \eqref{ep}, a triangle-inequality implies that
\begin{align}\nn
\|A^{\alpha}\uep(\cdot,t)\|_{L^2(\O)}&\le\|A^\alpha e^{-tA} u_0\|_{L^2(\Om)}+\int_0^t\|A^\alpha e^{-(t-s)A} \mathscr{P}(\uep\cdot\nabla)\uep(\cdot,s)\|_{L^2(\O)}ds\\\label{6.7.1}
&~~~~~~~~~~~+\int_0^t\|A^\alpha e^{-(t-s)A}\mathscr{P}\nep(\cdot,s)\nabla\phi\|_{L^2(\Om)}ds.
\end{align}
First we have
\begin{align}
\|A^\alpha e^{-tA} u_0\|_{L^2(\Om)}\le c_\alpha\|e^{-(t-1)A}u_0\|_{L^2(\O)}\le c_2 e^{-\lambda'_1(t-1)}\|u_0\|_{L^2(\Om)}\;\;\text{ for all }t>0.
\end{align}
Thanks to Lemma \ref{lem:nablau}, we know that $\|\nabla\uep(\cdot,t)\|_{L^2(\O)}\le c_2$ with some $c_2>0$, which together with
(\ref{fraction}), (\ref{interpolation}) and Lemma \ref{lem:integralestimates} yields the existence of $c_{\alpha}>0$ and $c_3>0$ such that 
\begin{align}\nn
&~~~~\int_0^t\|A^\alpha e^{-(t-s)A} \mathscr{P}(\uep\cdot\nabla)\uep(\cdot,s)\|_{L^2(\O)}ds\\\nn
&\le \int_0^tC_\alpha(t-s)^{-\alpha}e^{-\lambda_1'(t-s)}\|(\uep\cdot\nabla)\uep(\cdot,s)\|_{L^2(\O)}ds\\\nn
&\le \int_0^tC_\alpha(t-s)^{-\alpha}e^{-\lambda_1'(t-s)}\|\uep(\cdot,s)\|_{\Li}\|\nabla\uep(\cdot,s)\|_{L^2(\O)}ds\\\nn
&\le \int_0^t C_\alpha c_1c_2(t-s)^{-\alpha}e^{-\lambda_1'(t-s)}\|A^\alpha\uep(\cdot,s)\|_{L^2(\Om)}^a
\|\uep(\cdot,s)\|_{L^2(\O)}^{1-a}ds\\\nn
&\le\mathop{\sup}\limits_{t\in(0,T_{\max})}\norm[L^2(\O)]{\uep(\cdot,t)}^{1-a}
\int_0^t C_\alpha c_1c_2(t-s)^{-\alpha}e^{-\lambda_1'(t-s)}M^a(s)ds\\\label{6.7.2}
&\le C_\alpha c_1c_2c_3\mathop{\sup}\limits_{t\in(0,T_{\max})}\norm[L^2(\O)]{\uep(\cdot,t)}^{1-a}\mathop{\sup}\limits_{t\in(0,T)}M^a(t)
\end{align}
for all $t\in(0,T_{\max})$. Furthermore, by (\ref{fraction}) and Lemma \ref{lem:integralestimates} we can find and $c_4>0$ such that
\begin{align}\nn
&~~~~\int_0^t\|A^\alpha e^{-(t-s)A}\mathscr{P}\nep(\cdot,s)\nabla\phi\|_{L^2(\Om)}ds\\\nn
&\le \int_0^tC_\alpha\|\nabla\phi\|_{\Li}(t-s)^{-\alpha}e^{-\lambda'_1(t-s)}\|\nep(\cdot,s)\|_{L^2(\O)}ds\\\nn
&\le C_\alpha\|\nabla\phi\|_{\Li}\mathop{\sup}\limits_{t\in(0,T_{\max})}\norm[L^2(\O)]{\nep(\cdot,t)}\int_0^t (t-s)^{-\alpha}e^{-\lambda'_1(t-s)}ds\\\label{6.7.3}
&\le C_\alpha c_4\|\nabla\phi\|_{\Li}\mathop{\sup}\limits_{t\in(0,T_{\max})}\norm[L^2(\O)]{\nep(\cdot,t)}
\end{align}
for all $t\in(0,T_{\max})$. Taking supremum on both sides of (\ref{6.7.1}) on $(0,T)$ with $T\in(0,T_{\max})$, we use (\ref{6.7.2}) and (\ref{6.7.3}) to find $c_5>0$ such that
\begin{align}
\tilde{M}\le c_5+c_5\tilde{M}^a,
\end{align}
where we have used the notation $\tilde{M}:=\mathop{\sup}\limits_{t\in(0,T)}M(t)$. An application of Young's inequality on the above inequality leads to the assertion.
\end{proof}

\begin{proof}[Proof of Proposition \ref{epboundedness} (i)]
Let $p_0>2$ and let $\delta_0:=\delta(p_0)$ as defined in Lemma \ref{lem:Lp}. We immediately see from Lemmata \ref{lem:4.1}-\ref{lem:4.6} that
$\norm[L^\infty(\O)]{\nep(\cdot,t)}$ is bounded. The boundedness of $\norm[W^{1,q}(\O)]{\cep(\cdot,t)}$ is obvious from Lemma \ref{lem:nc} and Lemma \ref{lem:ns_nablac}. Also Lemma \ref{lem:uinfty} implies that $\norm[L^2(\O)]{A^\alpha \uep(\cdot,t)}$ is bounded. According to Lemme \ref{lem:localexistence}, we deduce $T_{\max}=\infty$, thus the solution is global.
\end{proof}
\section{Boundedness in three-dimensional case ($N=3$, $\kappa=0$)}

In this section, we deal with the chemotaxis-Stokes system in the three-dimensional setting. Since for the Navier-Stokes system, it is impossible to have global classical solutions without any restrictions on $u_0$, we only consider the case $\kappa=0$ and only assume $\norm[\Lin]{c_0}$ small.

We first give a sufficient condition for boundedness which in conjunction with Lemma \ref{lem:localexistence} proves Theorem \ref{th1}. In fact, since Lemma \ref{lem:Lp} provide $L^p$ estimate for any $p>1$, we can of course choose $p$ sufficiently large to get boundedness in $L^\infty(\Om)$. However, we would like to give an optimal condition in the following
for our own interest.

\begin{proposition}\label{prop:bddn}
Let $N=3$, $p>\frac N2$, suppose that
\begin{align}
\mathop{\sup}\limits_{t\in(0,T_{\max})}\|\nep(\cdot,t)\|_{L^p(\Om)}<\infty \;\;\text{ for all } t\in(0,T_{\max}).
\end{align}
Then we have for any $\ep>0$
\begin{align}
\mathop{\sup}\limits_{t\in(0,T_{\max})}\|\nep(\cdot,t)\|_{L^\infty(\Om)}<\infty \;\;\text{ for all } t\in(0,T_{\max}).
\end{align}
\end{proposition}

We will prove the proposition by several lemmata, which improve regularity for $\uep$ and $\nabla\cep$
in suitable way.

\begin{lem}\label{lem:s_bddu}
Let $p>\frac N2$, suppose that
\begin{align}
\mathop{\sup}\limits_{t\in(0,T_{\max})}\|\nep(\cdot,t)\|_{L^p(\Om)}<\infty.
\end{align} 
There are $\alpha\in(\frac N4,1)$ and $C>0$ such that
\begin{align*}
&\norm[L^2(\O)]{A^\alpha \uep(\cdot,t)}\le C,\\
&\norm[L^\infty(\O)]{\uep(\cdot,t)}\le C \text{ for all } t\in(0,T_{\max}).
\end{align*}
\end{lem}
\begin{proof}
The proof is very similar to Lemma \ref{lem:nablau}, we only need to deal with the term with less regularity of $\nep$, say, the case $p<2$. For $p\in(\frac N2,2)$ given in Lemma \ref{prop:bddn}, we can find $\alpha\in(\frac N4,\min\{1-\frac{N}{2p}+\frac N4,1\})$. We apply Lemma \ref{lem:integralestimates} to obtain some $c_2>0$ that
\begin{align}\nn
\int_0^t \|A^\alpha e^{(t-s)A}\nep\nabla \phi\|_{L^2(\O)} ds
&\le \int_0^t \|A^\alpha e^{\frac{(t-s)}{2}A}(e^{\frac{(t-s)}{2}A}\nep(\cdot,s)\nabla \phi)\|_{L^2(\O)} ds\\\nn
&\le \int_0^t C_\alpha(\frac{t-s}{2})^{-\alpha}e^{-\frac{\lambda_1'}{2}(t-s)}\|e^{\frac{(t-s)}{2}A}
\nep(\cdot,s)
\nabla \phi\|_{L^2(\O)} ds\\\nn
&\le \int_0^t C_\alpha(\frac{t-s}{2})^{-\alpha}e^{-\frac{\lambda_1'}{2}(t-s)} (\frac{t-s}{2})^{-\frac N2(\frac1p-\frac12)}
\|\nep(\cdot,s)\|_{L^{p}(\O)}\|\nabla\phi\|_{\Lin}ds\\\nn
&\le C_\alpha\int_0^t (t-s)^{-\alpha-\frac{N}{2p}+\frac N4} e^{-\frac{\lambda_1'}{2}(t-s)}\|\nep(\cdot,s)\|_{L^{p}(\O)}\|\nabla\phi\|_{\Lin}ds\\\nn
&\le C_\alpha\|\nabla\phi\|_{\Lin}\mathop{\sup}\limits_{t\in(0,T_{\max})}\|\nep(\cdot,t)\|_{L^p(\Om)}
\int_0^t (t-s)^{-\alpha-\frac{N}{2p}+\frac N4} e^{-\frac{\lambda_1'}{2}(t-s)}ds\\\nn
&\le C_\alpha c_2\|\nabla\phi\|_{\Lin}\mathop{\sup}\limits_{t\in(0,T_{\max})}\|\nep(\cdot,t)\|_{L^p(\Om)},
\end{align}
for all $t\in(0,T_{\max})$. Combine this fact with the proof of Lemma \ref{lem:boundu} we see that \(\norm[L^2(\O)]{A^\alpha\uep(\cdot,t)}\) is bounded for all $t\in(0,T_{\max})$. Since $\alpha>\frac N4$, embedding theorem implies the boundedness of \(\norm[\Lin]{\uep(\cdot,t)}\). Thus the proof is complete.
\end{proof}

\begin{lem}\label{lem:s_nablac}
Let $p>\frac N2$, suppose that
\begin{align}
\mathop{\sup}\limits_{t\in(0,T_{\max})}\|\nep(\cdot,t)\|_{L^p(\Om)}<\infty.
\end{align}
For $N<q_0<\frac{Np}{N-p}$, there is $C>0$ such that for any $\ep>0$,
\begin{align}
\|\nabla\cep(\cdot,t)\|_{L^{q_0}(\O)}\le C\;\; \text{ for all } t\in(0,T_{\max}).
\end{align}
\end{lem}
\begin{proof}
Let $t\in(0,T_{\max})$ and $M:=\mathop{\sup}\limits_{t\in(0,T)}\|\nabla\cep(\cdot,t)\|_{L^{q_0}(\O)}$.
The variation of constants formula implies
\begin{align}\nn
\|\nabla\cep(\cdot,t)\|_{L^{q_0}(\Om)}&\le \|\nabla e^{t\Delta}c_0\|_{L^{q_0}(\Om)}+\int_0^t\|\nabla e^{(t-s)\Delta}(\nep\cep)(\cdot,s)\|_{L^{q_0}(\Om)}ds \\\nn
&~~~~~~~~~~~~~~+\int_0^t\|\nabla e^{(t-s)\Delta}(\uep\cdot\nabla\cep)(\cdot,s)\|_{L^{q_0}(\Om)}ds
\end{align}
for all $t\in(0,T_{\max})$.
By the $L^p$-$L^q$ estimates we see that
\begin{align}\nn
\|\nabla e^{t\Delta}c_0\|_{L^{q_0}(\Om)}\le c_1 t^{-\frac12}\|c_0\|_{L^{q_0}(\O)}
\end{align}
for some $c_1>0$ and for all $t>0$.
Since $1<q_0<\frac{Np}{N-p}$, we know that $-\frac12-\frac N2(\frac 1p-\frac 1{q_0})>-1$, thus we estimate the second term by $L^p$-$L^q$ estimate for Neumann semigroup with $c_2>0$ such that
\begin{align}\nn
&\int_0^t\|\nabla e^{(t-s)\Delta}\nep(\cdot,s)\cep(\cdot,s)\|_{L^{q_0}(\Om)}ds\\\nn
&\le \int_0^t c_2(1+(t-s)^{-\frac12-\frac N2(\frac 1p-\frac 1{q_0})})e^{-\lambda_1(t-s)}\|\nep(\cdot,s)\cep(\cdot,s)\|_{\Lp}ds\\\nn
&\le \int_0^t c_2(1+(t-s)^{-\frac12-\frac N2(\frac 1p-\frac 1{q_0})})e^{-\lambda_1(t-s)}\|\nep(\cdot,s)\|_{\Lp}\|\cep(\cdot,s)\|_{\Li}ds
\end{align}
for all $t\in(0,T_{\max})$. It is bounded due to the choice of $q_0$ and our assumption on $\norm[L^p(\Om)]{\nep(\cdot,t)}$.
Now we fix $q<q_0$ satisfying $\frac 1q-\frac 1{q_0}<\frac 1N$, and let $a=\frac{1-\frac Nq}{1-\frac N{q_0}}\in(0,1)$. H\"older's inequality as well as interpolation inequality yield the existence of $c_3>0$,
\begin{align}\nn
&\int_0^t\|\nabla e^{(t-s)\Delta}\uep(\cdot,s)\cdot\nabla\cep(\cdot,s)\|_{L^{q_0}(\Om)}ds\\\nn
&\le\int_0^t c_2(1+(t-s)^{-\frac12-\frac N2(\frac 1q-\frac 1{q_0})})e^{-\lambda_1(t-s)}\|\uep(\cdot,s)\nabla\cep(\cdot,s)\|_{\Lq}ds\\\nn
&\le \int_0^t c_2(1+(t-s)^{-\frac12-\frac N2(\frac 1q-\frac 1{q_0})})e^{-\lambda_1(t-s)}\|\nabla\cep(\cdot,s)\|_{\Lq}\|\uep(\cdot,s)\|_{\Li}ds\\\nn
&\le \int_0^t c_2(1+(t-s)^{-\frac12-\frac N2(\frac 1q-\frac 1{q_0})})e^{-\lambda_1(t-s)}\|\uep(\cdot,s)\|_{\Li}(c_3\|\nabla\cep(\cdot,s)\|_{L^{q_0}(\O)}^a
\|\cep(\cdot,s)\|_{\Li}^{1-a}\\\nn
&~~~~~~~~~~~~~~~~~~~~~~~~~~~~~~~~~~~~~~~~~~~+c_3\|\cep(\cdot,s)\|_{\Li})ds\\\nn
&\le c_4M^a+c_4.
\end{align}
for all $t\in(0,T_{\max})$ with some $c_4>0$.
The assertion can be seen by combining the above estimates and the fact that $a<1$.
\end{proof}

Having enough regularity for both $\uep$ and $\nabla\cep$, we are ready to prove boundedness for $\nep$.
\begin{proof}[Proof of Proposition \ref{prop:bddn}]
Let $T\in(0,T_{\max})$ and $M:=\mathop{\sup}\limits_{t\in(0,T)}\|\nep(\cdot,t)\|_{\Li}$.
The representation formula for $\nep$ yields that
\begin{align}
\|\nep(\cdot,t)\|_{\Li}&\le \|e^{t\Delta}n_0\|_{\Li}+\int_0^t\|e^{(t-s)\Delta}\nabla\cdot(\nep S_\ep(\cdot,\nep,\cep)\nabla\cep)(\cdot,s)\|_{\Li}ds\\\nn
&~~~~~~+\int_0^t
\|e^{(t-s)\Delta}\nabla\cdot(\nep(\cdot,s)\uep(\cdot,s))\|_{\Li}ds
\end{align}
for all $t\in(0,T_{\max})$.
Since $q_0>N$, we can find $N<p_0<q_0$ and $q_1>1$ such that $\frac{1}{p_0}=\frac{1}{q_0}+\frac{1}{q_1}$. Let $a=1-\frac{1}{q_1}$. $L^p$-$L^q$ estimates for the Neumann heat semigroup and H\"older interpolation imply $c_1>0$ and $c_2>0$
\begin{align}\nn
&\int_0^t\|e^{(t-s)\Delta}\nabla\cdot(\nep S_\ep(\cdot,\nep,\cep)\cdot\nabla\cep)(\cdot,s)\|_{\Li}ds\\\nn
&\le \int_0^t c_1(1+(t-s)^{-\frac 12-\frac N{2p_0}})e^{-\lambda_1(t-s)}\|(\nep S_\ep(\cdot,\nep,\cep)\cdot\nabla\cep)(\cdot,s)\|_{L^{p_0}(\O)}ds\\\nn
&\le \int_0^t c_1S_0(1+(t-s)^{-\frac 12-\frac N{2p_0}})e^{-\lambda_1(t-s)} \|\nabla\cep(\cdot,s)\|_{L^{q_0}(\O)}\|\nep(\cdot,s)\|_{L^{q_1}(\O)}ds\\\nn
&\le \int_0^t c_1S_0(1+(t-s)^{-\frac 12-\frac N{2p_0}})e^{-\lambda_1(t-s)} \|\nabla\cep(\cdot,s)\|_{L^{q_0}(\O)}\|\nep(\cdot,s)\|_{\Li}^a \|\nep(\cdot,s)\|_{L^1(\O)}ds\\\nn
&\le c_2+c_2 M^a, \text{ for all } t\in(0,T_{\max}).
\end{align}
Now we pick $p_1>N$, and let $b=1-\frac{1}{p_1}$. The $L^p-L^q$ estimate and the interpolation inequality imply
\begin{align}
&\int_0^t
\|e^{(t-s)\Delta}\nabla\cdot(\nep(\cdot,s)\uep(\cdot,s))\|_{\Li}ds\\\nn
&\le \int_0^t
c_1(1+(t-s)^{-\frac12-\frac N{2p_1}})e^{-\lambda_1(t-s)}\|\nep(\cdot,s)\uep(\cdot,s)\|_{L^{p_1}(\Om)}ds\\\nn
&\le \int_0^t c_1(1+(t-s)^{-\frac12-\frac N{2p_1}})e^{-\lambda_1(t-s)}\|\nep(\cdot,s)\|_{L^{p_1}(\Om)}
\|\uep(\cdot,s)\|_{\Li}ds\\\nn
&\le \int_0^t c_1(1+(t-s)^{-\frac12-\frac N{2p_1}})e^{-\lambda_1(t-s)}
\|\uep(\cdot,s)\|_{\Li} \|\nep(\cdot,s)\|_{\Li}^b\|\nep(\cdot,s)\|_{L^1(\O)}^{1-b}ds.
\end{align}
Finally, collecting the above estimates,  we conclude the assertion by a similar reasoning as in Lemma \ref{lem:s_nablac}.
\end{proof}

\subsection{Proof of Proposition \ref{epboundedness} (ii)}
Combining Proposition \ref{prop:bddn} and Lemma \ref{lem:Lp} proves Proposition \ref{epboundedness}.
\begin{proof}[Proof of Proposition \ref{epboundedness} (ii)]
Let $p>2$ and let $\delta_0:=\delta_0(p)$ as defined in Lemma \ref{lem:Lp}. We see that (\ref{c0}) implies $\norm[L^p(\Om)]{n(\cdot,t)}$ is bounded, which combined with Proposition \ref{prop:bddn} yields the boundedness of $\norm[\Lin]{\nep(\cdot,t)}$. The boundedness of $\norm[\Lin]{\nabla\cep(\cdot,t)}$ has been shown in Lemma \ref{lem:s_nablac}. Together with Lemma \ref{lem:s_bddu}, we see that the solution is global by Lemma \ref{lem:localexistence}.
\end{proof}

\section{Passing to the limit}

We now wish to obtain the solution of (\ref{eq}) by sending $\ep\to 0$ for the approximated solution.
In order to achieve this, we shall first prepare some estimates for $(\nep,\cep,\uep)$ which are independent of $\ep$. Since we cannot expect the regularity in $C^{2+\alpha,1+\frac{\alpha}{2}}
(\overline{\O}\times(0,\infty))$ to be uniform in $\ep$ due to the presence of $S_\ep$, we will first show the triple of limit functions solves \eqref{eq} in the sense of distributions, then apply standard parabolic regularity to show that it is actually a classical solution. The procedure is quite similar to that in \cite{cao_lan}.\\

Let us first define a weak solution.
\begin{dnt}\label{def:1}
We say that $(n,c,u,P)$ is a global weak solution of (\ref{eq}) associated to initial data $(n_0,c_0,u_0)$ if
\begin{align}
\left\{
\begin{array}{llc}
 n\in L^\infty((0,\infty)\times\Om)\cap L^2_{loc}((0,\infty);W^{1,2}(\Om)),\\[6pt]
  c\in  L^\infty((0,\infty)\times\Om)\cap L^2_{loc}((0,\infty);W^{1,2}(\Om)),\\[6pt]
   u\in L^\infty((0,\infty)\times\Om)\cap L^2_{loc}((0,\infty);W_{0,\sigma}^{1,2}(\Om)),\\[6pt]
   P\in L^2((0,T); W^{1,2}(\Om)),
\end{array}
\right.
\end{align}
and for all $\psi\in C_0^\infty(\overline{\O}\times[0,\infty);\mathbb{R})$ and all $\zeta\in C_{0,\sigma}^\infty(\Om\times[0,\infty);\mathbb{R}^N)$ the following identities hold:
\begin{align}\nn
\displaystyle
-\int_0^\infty\intO n\psi_t-\intO n_0\psi(\cdot,0)&=-\int_0^\infty\intO \nabla n\cdot\nabla\psi
\\\label{weaksol}
&~~~~~~~~~~~~~~~~+\int_0^\infty\intO nS(x,n,c)\cdot\nabla c\cdot\nabla\psi+\int_0^\infty\intO nu\cdot\nabla\psi,\\[6pt]
\displaystyle
-\int_0^\infty\intO c\psi_t-\intO c_0\psi(\cdot,0)&=-\int_0^\infty\intO \nabla c\cdot\nabla\psi
-\int_0^\infty\intO nc\psi+\int_0^\infty\intO cu\cdot\nabla\psi,\\[6pt]
\displaystyle
-\int_0^\infty\intO u\cdot\zeta_t-\intO u_0\cdot\zeta(\cdot,0)&=-\int_0^\infty\intO \nabla u\cdot\nabla\zeta
+\int_0^\infty\intO (u\cdot\nabla) u\cdot\zeta+\int_0^\infty\intO n\nabla \phi\cdot\zeta.
\end{align}
\end{dnt}

The required estimates are very close to those in \cite{cao_lan}. We will state the results here and only give a sketch of the proofs.

\begin{lem}\label{lem:estimate1}
There exists $C>0$ such that for all $\ep\in(0,1)$
\begin{align}
\label{est:nablac}
&\int_0^\infty\intO |\nabla \cep|^2<C,\\
\label{est:nablan}
&\int_0^\infty\intO |\nabla \nep|^2<C.
\end{align}
\end{lem}

\begin{proof}
Multiply the second equation with $\cep$, using the fact $\nabla\cdot \uep=0$, we obtain that
\begin{align}
\frac{1}{2}\frac{d}{dt}\intO \cep^2+\intO |\nabla \cep|^2\le 0,
\end{align}
this implies (\ref{est:nablac}) by direct integration. Since $T_{\max}=\infty$, letting $p=2$
in (\ref{3:2}), we see that (\ref{est:nablan}) holds.
\end{proof}



\begin{lem}\label{lem:regularity1}
All bounded solution $(\nep,\cep,\uep,P_\ep)$ of (\ref{ep}) satisfy
\begin{align}\label{holder}
  \nep\in C^{\gamma,\frac\gamma2}_{loc}(\overline{\Om}\times(0,\infty)),\;
  \cep\in C^{\gamma,\frac\gamma2}_{loc}(\overline{\Om}\times(0,\infty)),\;
  \uep\in C^{1+\gamma,\gamma}_{loc}(\overline{\Om}\times(0,\infty)).
 \end{align}
More precisely, there is $C>0$ such that for all $\ep\in(0,1)$, and all $s\in[1,\infty)$ we have
\begin{align}\label{holder:n}
&\|\nep\|_{C^{\gamma,\frac{\gamma}{2}}(\overline{\O}\times[s,s+1])}\le C,\\\label{holder:c}
&\|\cep\|_{C^{\gamma,\frac{\gamma}{2}}(\overline{\O}\times[s,s+1])}\le C,\\\label{holder:u}
&\|\uep\|_{C^{1+\gamma,\gamma}(\overline{\O}\times[s,s+1])}
\le C.
\end{align}
\end{lem}
\begin{proof}
Let $s\ge 1$. We define $\tilde{n}_\ep(\cdot,t)=\nep(\cdot,t+s-1)$, $\tilde{c}_\ep(\cdot,t)=\cep(\cdot,t+s-1)$ and $\tilde{u}_\ep(\cdot,t)=\uep(\cdot,t+s-1)$. Let $\xi\in C^\infty((0,\infty))$ satisfy $\xi=0$ on $(0,\frac12)\cup(\frac52,\infty)$ and $\xi=1$ on $[1,2]$. We see that $\xi\tilde{n}_\ep$ is a weak solution of 
\[
(\xi\tilde{n}_\ep)_t-\nabla\cdot(\nabla(\xi\tilde{n}_\ep)-\xi\tilde{n}_\ep S_\ep(x,\tilde{n}_\ep,\tilde{c}_\ep)\nabla\tilde{c}_\ep)=\xi'\tilde{n}_\ep,\;\; t\in[0,\infty),\]
 associated with Neumann boundary condition and $\xi\tilde{n}_\ep(\cdot,\frac12)=0$.
Since \((\nabla(\xi\tilde{n}_\ep)-\xi\tilde{n}_\ep S_\ep(x,\tilde{n}_\ep,\tilde{c}_\ep)\nabla\tilde{c}_\ep-\tilde{n}_\ep\tilde{u}_\ep)\cdot\nabla(\xi\tilde{n}_\ep)> \frac12|\nabla(\xi\tilde{n}_\ep)|^2-\tilde{n}_\ep^2|S_\ep|^2|\nabla\tilde{c}_\ep|^2-\tilde{n}_\ep^2|\tilde{u}_\ep|^2\), we see together with the fact guaranteed in Lemma \ref{lem:ns_nablac} and Lemma \ref{lem:s_nablac},  that the norms of $\tilde{n}_\ep^2|S_\ep|^2|\nabla\tilde{c}_\ep|^2+\tilde{n}_\ep^2|\tilde{u}_\ep|^2$ and $\xi'\tilde{n}_\ep$ are bounded in $L^p((\frac12,\frac52);L^q(\O))$ for suitably large $p$ or $q$ and independent of $s$, thus Theorem 1.3 in \cite{holder} implies there is $\gamma_1\in(0,1)$ and $c_1>0$ such that
\[ \norm[ {C^{\gamma_1,\frac{\gamma_1}{2}} (\overline{\O}\times[s,s+1])}]{\nep}=\norm[ {C^{\gamma_1,\frac{\gamma_1}{2}} (\overline{\O}\times[1,2])}]{\tilde{n}_\ep}
\le\norm[ {C^{\gamma_1,\frac{\gamma_1}{2}} (\overline{\O}\times[\frac12,\frac52])}]{\xi\tilde{n}_\ep}\le c_1,\]
and $c_1$ depends on $\norm[L^\infty({\Om\times(\frac12,\frac25)})]{\xi\tilde{n}_\ep}$ and the norms of $\tilde{n}_\ep^2|S_\ep|^2|\nabla\tilde{c}_\ep|^2+\tilde{n}_\ep^2|\tilde{u}_\ep|^2$ in appropriate spaces only. A similar reasoning yields some $\gamma_2\in(0,1)$ and $c_2>0$ such that
\begin{align*}
\norm[C^{\gamma_2,\frac{\gamma_2}{2}}
{(\overline{\O}\times[s,s+1])}]
{\cep}\le c_2.
\end{align*}
The derivation of the regularity of $\uep$ is similar to \cite[Lemma 5.3]{cao_lan}.
Let $s\ge 1$, and $\xi_s$ be a smooth function: $(0,\infty)\to [0,1]$ satisfying $\xi_s(t)=0$ on $(0,s-\frac12)\cup(s+\frac32,\infty)$ and $\xi_s(t)=1$ on $[s,s+1]$. We consider $\xi\cdot\uep$, it satisfies
\[
(\xi\uep)_t=\xi_t \uep+\xi {\uep}_t=\Delta (\xi\uep)-\xi(\uep\cdot\nabla)\uep+\xi\nep\nabla\phi
+\xi'\uep, \mbox{ on }(s-\frac12,s+\frac32),
\]
with $\xi\uep(\cdot,0)=0$ and $\xi \uep=0$ on $\partial \O$. Thus by an application of \cite[Thm. 2.8]{giga_sohr}, for any $r\in(1,\infty)$, we deduces the existence of constant $C_r>0$ fulfilling
\[
\int_0^\infty\norm[L^r(\O)]{(\xi\uep)_t}^r
+\int_0^\infty\norm[L^r(\O)]{D^2(\xi\uep)}^r
\le C_r\left(0+
\int_0^\infty \norm[L^r(\O)]{\mathscr P \left((\xi\uep\cdot\nabla)\uep\right)+\mathscr P\left(\xi\nep\nabla\phi\right)+\mathscr P\left(\xi'\uep\right)}^rds
\right),
\]
which, due to the definition of $\xi$  and boundedness of $\uep$, $\nep$, $\xi'$ reads as
\[
\int_{s-\frac12}^{s+\frac32}
\norm[L^r(\O)]{(\xi\uep)_t}^r
+\int_{s-\frac12}^{s+\frac32}
\norm[L^r(\O)]{D^2(\xi\uep)}^r
\le C_1 \int_{s-\frac12}^{s+\frac32}
\norm[L^r(\O)]{\nabla (\xi\uep)}^r+C_2
\]
 for all $s\in(1,\infty)$ and for some $C_1>0$, $C_2>0$.
Let $a=\frac{1-\frac Nr}{2-\frac Nr}\in (0,1)$, the Gargliardo-Nirenberg inequality shows
\[\norm[L^r(\O)]{\nabla(\xi\uep)}^r\le C_3\norm[L^r(\O)]{D^2(\xi\uep)}^{ar}
\norm[\Lin]{(\xi\uep)}^{(1-a)r}.
\]
for some $C_3>0$.
Integrating the above inequality on $(s-\frac12,s+\frac32)$ and using Young's inequality yields that
\[
\int_{s-\frac12}^{s+\frac32}\norm[L^r(\O)]
{\nabla(\xi\uep)}^r\le C_4
\int_{s-\frac12}^{s+\frac32}
\norm[L^r(\O)]{D^2(\xi\uep)}^{ar}\\
\le\int_{s-\frac12}^{s+\frac32}
\left(\frac12\norm[L^r(\O)]{D^2(\xi\uep)}^r+C_5
\right).
\]
with some $C_4>0$, $C_5>0$ and for all $s\in[1,\infty)$.
Combining the above estimates we see that there is $C_6>0$ such that for all $s\ge 1$ and all $\ep\in(0,1)$,
\[
\int_s^{s+1}\norm[L^r(\O)]{(\uep)_t}^r+
\int_s^{s+1}\norm[L^r(\O)]{D^2\uep}^r\le C_6.
\]
Let $r\in(1,\infty)$ be sufficiently large, the embedding theorem implies the existence of $\gamma_3\in(0,1)$, $C>0$ such that
\begin{align*}
\norm[C^{1+\gamma_3,{\gamma_3}}{(\overline{\O}
\times[s,s+1])}]{\uep}\le C.
\end{align*}

Choosing $\gamma=\min\{\gamma_1,\gamma_2,\gamma_3\}$ we have proved (\ref{holder:n}-\ref{holder:u}). For all $\tau>0$, if we choose  $\xi_\tau\in C_0^\infty((0,\infty))$ in such way that $\xi_\tau=0$ on $(0,\tau)$ and $(\max\{3\tau,1\},\infty)$, and $\xi_\tau=1$ on $[\tau,\max\{2\tau,1\}]$. We consider the equations for $\xi_\tau\nep$, $\xi_\tau\cep$ and $\xi_\tau\uep$, then \eqref{holder} is obtained by the same reasoning as above. Actually, the way of $\gamma_i$ ($i=1,2$) depending on $\tau$ is through the non-decreasing dependence of the norms $\|{\xi\nep}\|_{L^\infty(\O\times[\tau,\max\{3\tau,1\}])}$, $\|{\xi\cep}\|_{L^\infty(\O\times[\tau,\max\{3\tau,1\}])}$ \cite[Theorem 1.3]{holder}, which are independent of $\tau$, thus we can choose the same $\gamma_i$ (i=1,2) as before. Moreover, $\gamma_3$ can be chosen in the same manner upon choosing the same $r$.

\end{proof}

\begin{lem}\label{lem:convergence}
Let $\gamma\in(0,1)$ be chosen as in Lemma \ref{lem:regularity1}. There exists $(\ep_j)_{j\in\mathbb{N}}\subset(0,1)$ such that $\ep_j\searrow 0$ as $j\to\infty$, and that as $\ep=\ep_j\searrow 0$, it holds that
\begin{align}
\label{con:nae}
 \nep\rightarrow n &\text{ in } C_{loc}^\gamma(\overline{\Om}\times(0,\infty)),\\
\label{con:nablan}
 \nabla\nep\rightharpoonup \nabla n &\text{ in }L^2(\Om\times(0,\infty)),\\
\label{con:cae}
 \cep\rightarrow c &\text{ in } C_{loc}^\gamma(\overline{\Om}\times(0,\infty)),\\
 \label{con:nablac}
 \nabla\cep\rightharpoonup \nabla c &\text{ in }L^2(\Om\times(0,\infty)),\\
\label{con:uae}
 \uep\rightarrow u &\text{ in } C_{loc}^\gamma(\overline{\Om}\times(0,\infty)),\\
 \label{con:nablau}
 \nabla\uep\rightarrow \nabla u &\text{ in }C_{loc}^\gamma(\Om\times(0,\infty)),\\
 \label{con:u}
 \uep\wstarto u &\text{ in }L^\infty((0,\infty);D(A^\alpha)),\\
\label{con:S}
 S_\ep(x,\nep(x,t),\cep(x,t)) \rightarrow S(x,n(x,t),c(x,t)) &\text{ a.e. in }  \Om\times(0,\infty).
\end{align}
\end{lem}
\begin{proof}
First (\ref{con:nae}), (\ref{con:cae}) and (\ref{con:uae}), (\ref{con:nablau}) are obtained from  Lemma \ref{lem:regularity1}. Lemma \ref{lem:estimate1} implies (\ref{con:nablan}) and (\ref{con:nablac}). Due to the obtained convergence \eqref{con:nae} and \eqref{con:cae} and the continuity of $S$, we conclude that (\ref{con:S}) holds.
\end{proof}

\begin{lem}\label{lem:isweaksol}
 The functions $n, c, u$ from Lemma \ref{lem:convergence} form a weak solution to \eqref{eq} in the sense of Definition \ref{def:1}.
\end{lem}
\begin{proof}
Take $\psi$ and $\zeta$ as specified in Definition \ref{def:1} and test them to (\ref{ep}). Lemma \ref{lem:convergence} allow us to take limit in each integral, thus we obtain the weak formulation.
\end{proof}

\begin{lem}\label{lem:regularity2}
 The functions $n,c,u$ from the previous lemma satisfy
\begin{align}\label{classical_regularity}
  n\in C^{2+\gamma,1+\frac\gamma2}_{loc}(\overline{\Om}\times(0,\infty)),\;
  c\in C^{2+\gamma,1+\frac\gamma2}_{loc}(\overline{\Om}\times(0,\infty)),\;
  u\in C^{2+\gamma,1+\frac\gamma2}_{loc}(\overline{\Om}\times(0,\infty))
 \end{align}
 for some $\gamma\in(0,1)$. Moreover, let $s\ge 1$.There is a constant $C>0$ such that
 \begin{align}\label{c2+al}
 \norm[{C^{2+\gamma,1+\frac\gamma2}(\overline{\Om}\times[s,s+1])}]{n}\le C,\;
 \norm[{C^{2+\gamma,1+\frac\gamma2}(\overline{\Om}\times[s,s+1])}]{c}\le C,\;
 \norm[{C^{2+\gamma,1+\frac\gamma2}(\overline{\Om}\times[s,s+1])}]{u}\le C
 \end{align}
 for all $t\ge 1$.
\end{lem}
\begin{proof}
First taking $\xi_s$ as chosen in Lemma \ref{lem:regularity1}, we see that $\xi \cdot c$ is a weak solution of $(\xi c)_t-\Delta (\xi c)+n(\xi c)+u\cdot\nabla (\xi c)-\xi' c=0$ on $t\in(s-\frac12,s+\frac32)$ associated with Neumann boundary condition and $\xi c(\cdot,s-\frac12)=0$. First \cite[Thm 5.3]{lady} grantees that $\xi c\in C^{2+\gamma,1+\frac{\gamma}{2}}(\overline{\O}\times[s-\frac12,s+\frac32])$, therefore, \cite[Thm 4.9]{liberman} shows that the norm \(\norm[{C^{2+\gamma,1+\frac{\gamma}{2}}(\overline{\O}\times[s-\frac12,s+\frac32])}]{\xi c}\) is controlled by the
corresponding H\"older norms of $n$ and $u$, which is bounded by Lemma \ref{lem:regularity1} and Lemma \ref{lem:convergence}.

For the regularity of $n$, we improve it similarly as $c$ but more carefully since its boundary condition also involving $c$. We first estimate its $C^{1+\gamma,\frac{1+\gamma}{2}}$ norm, then its $C^{2+\gamma,1+\frac{\gamma}{2}}$ norm, by  \cite[Thm 4.8]{liberman} and \cite[Thm 4.9]{liberman}, respectively.

For the regularity of $u$, we again consider $\xi\cdot u$, which satisfies \((\xi u)_t=\Delta (\xi u)-\xi(u\cdot\nabla)u+\xi n\nabla\phi+\xi'u,\) with Dirichlet boundary condition. Lemma \ref{lem:estimate1} already ensures the $C^{\gamma,\frac{\gamma}{2}}$ bound on the right hand side. Thus \cite[Thm 1.1]{solonnikov} together with the uniqueness guaranteed in \cite[Thm.V.1.5.1]{sohr} implies the \eqref{c2+al}.

For any fixed $\tau>0$, by choosing $\xi_\tau\in C_0^\infty((0,\infty))$ such that $\xi_\tau=0$ on $(0,\tau)$ and $(3\tau,\infty)$, and $\xi_\tau=1$ on $[\tau,2\tau]$, we can see \eqref{classical_regularity} holds by the same reasoning as introduced above.
\end{proof}

Having in hand the regularity for the weak solution \((n,c,u)\)  of (\ref{eq}), we have shown that it is actually classical solution.

\section{Stabilization}
In this section, we prove large time convergence for each component of the solution but not the approximated one.
Since we have already derived that the solution is globally bounded, then has uniform in time regularity in H\"older space. In order to make the convergence property from the previous section applicable, we have to gain some uniform smooth regularity.

The first obtained convergence of $n$ is crucial, it will imply convergence for $c$ and $u$ later.
\begin{lem}\label{lem:nablan}
Let $(n,c,u,P)$ be the classical bounded solution of (\ref{eq}). We have
\[
\int_0^\infty\intO |\nabla n|^2<\infty.
\]
\end{lem}
\begin{proof}
The statement holds due to (\ref{est:nablan}) and (\ref{con:nablan}).
\end{proof}
\begin{lem}\label{lem:con:n}
Let $(n,c,u,P)$ be the classical solution of (\ref{eq}), we have
\begin{align}\label{ninfty'}
\|(n(\cdot,t)-\bar{n}_0)\|_{\Li}\to 0,\text{ as }t\to\infty.
\end{align}
\end{lem}
\begin{proof}
Suppose on contrary that there are $c_1>0$ and a sequence $t_k\to\infty$ such that
\begin{align}\label{contract}
\norm[L^\infty(\O)]{n(\cdot,t_k)-\bar{n}_0}>c_1\;\;\mbox{ for } k\in \mathbb{N}.
\end{align}
Now we define
\[g_k(x,s)=n(x,s+t_k),\;(x,s)\in\overline{\O}\times[0,1].
\]
By the regularity guaranteed in Lemma \ref{lem:regularity2}, we see that for all $k\in \mathbb{N}$, there is $c_2>0$ such that
\[
\norm[C^{2+\alpha,1+\frac{\alpha}{2}}{(\overline{\O}\times[0,1])}]
{g_k}\le c_2.
\]
The Arzel\'a-Ascoli theorem implies that $g_k$ is relatively compact in $C^1(\overline{\O}\times[0,1])$. Thus we can find a subsequence $\{g_{k_j}\}_j$ and $n_\infty\in C^1(\overline{\O}\times[s,s+1])$ such that

\begin{align}\label{g_kj}
g_{k_j}\to n_\infty \in C^1(\overline{\O}\times[0,1]), \mbox{ as }j\to\infty.
\end{align}
It is left to show $n_\infty=\bar{n}_0$.
We see from Lemma \ref{lem:nablan} that
\[
\int_0^1\intO |\nabla g_{k_j}|^2\to 0 \mbox{ as } j\to \infty,
\]
which combined with (\ref{g_kj}) implies
\[
\int_0^1\intO |\nabla n_\infty|^2= 0.
\]
Since $n_\infty\in C^1(\overline{\O}\times[0,1])$, we deduce that \(n_\infty\equiv L \) with $L\in\mathbb{R}$. Moreover, we have
 \[|\O|\cdot L=\int_0^1\intO n_\infty=\mathop{\lim}\limits_{j\to\infty}
\int_0^1\intO f_{k_j}=\int_0^1\intO n_0.\]
Thus we conclude \(n_\infty\equiv \bar{n}_0\). This contradicts (\ref{contract}) by the definition of $n_\infty$.
\end{proof}

\begin{lem}\label{lem:con:c}
Let $(n,c,u,P)$ be the classical bounded solution of (\ref{eq}).
For any $0<\eta<\bar{n}_0$, there is $C>0$ such that
\begin{align}
\|c(\cdot,t)\|_{\Li}\le C e^{-\eta t}\;\; \text{ for all } t\ge 0.
\end{align}
\end{lem}
\begin{proof}
For all $0<\eta<\bar{n}_0$, we can find $T>0$ such that
\begin{align}\nn
n\ge \eta\;\; \text{ for all } t\ge T.
\end{align}
Thus the second equation of (\ref{eq}) can be written as
\begin{align}\nn
{c}_t\le \Delta c-\eta c-u\cdot\nabla c\;\;\text{ for all } t\ge T.
\end{align}
The maximum principle yields that
\begin{align}\nn
\|c(\cdot,t)\|_{\Li}\le \|c(\cdot,T)\|_{\Li} e^{-\eta t}\;\; \text{ for all } t\ge T.
\end{align}
An obvious choice of $C$ completes the proof.
\end{proof}

\begin{lem}\label{lem:con:nablac}
Let $(n,c,u,P)$ be the classical bounded solution of (\ref{eq}), and let $0<\eta<\min\{\bar{n}_0,\lambda_1\}$, then we can find some $C>0$ such that
\begin{align}
\|\nabla c(\cdot,t)\|_{\Li}\le C e^{-\eta t}\;\;\text{ for all } t\ge 0.
\end{align}
\end{lem}
\begin{proof}
The proof is based on the constants variation formula, $L^p-L^q$ estimates as well as Lemma \ref{lem:integralestimates}.
Recall \eqref{5.1.1}, the first term can be estimated easily
\begin{align}\nn
\|\nabla e^{t\Delta}c_0\|_{L^\infty(\Om)}&\le c_1 t^{-\frac{1}{2}} e^{-\lambda_1 t}\|c_0\|_{\Li}\\\nn
&\le c_1\|c_0\|_{\Li} e^{-\lambda_1 t}
\end{align}
for all $t\ge 1$. With local existence theory, $\norm[\Lin]{\nabla e^{t\Delta}c_0}$ is bounded  for $t>0$. Similar to \eqref{5.1.3}, with
$\norm[L^p(\Om)]{n}\le c_2$ and $\norm[\Lin]{c}\le c_3$, if we choose $p>\frac1N$, Lemma \ref{lem:integralestimates} implies the existence of
$c_4>0$ such that
\begin{align}\nn
\int_0^t\|\nabla e^{(t-s)\Delta} n(\cdot,s)c(\cdot,s)\|_{L^\infty(\Om)}ds
&\le \int_0^t c_1(1+(t-s)^{-\frac12-\frac N{2p}})e^{-\lambda_1(t-s)}\|n(\cdot,s)\|_{L^{p}(\Om)}
\|c(\cdot,s)\|_{L^\infty(\Om)}ds\\\nn
&\le  \int_0^t c_1(1+(t-s)^{-\frac12-\frac N{2p}})e^{-\lambda_1(t-s)}c_2 c_3 e^{-\eta s}ds\\\nn
&\le c_1c_2c_3c_4e^{-\eta t}
\end{align}
for all $t>0$. Moreover, let $\theta=1-\frac{N}{p}\in(0,1)$ and $M(t)= e^{\eta t}\|\nabla\cep(\cdot,t)\|_{\Li}$,
\begin{align}\nn
&~~~~\int_0^t\|\nabla e^{(t-s)\Delta}(u(\cdot,t)\cdot\nabla c(\cdot,t))\|_{L^\infty(\Om)}ds\\\nn
&\le \int_0^t c_1(1+(t-s)^{-\frac12-\frac N{2p}})e^{-\lambda_1(t-s)}\|u(\cdot,s)\cdot\nabla c\|_{L^p(\O)}ds\\\nn
&\le \int_0^t c_1(1+(t-s)^{-\frac12-\frac N{2p}})e^{-\lambda_1(t-s)}\|u(\cdot,s)\|_{\Li}
(\|\nabla c(\cdot,t)\|_{\Li}^\theta
\|c(\cdot,t)\|_{L^\infty(\Om)}^{1-\theta}+\|c(\cdot,s)\|_{\Li})ds\\\nn
&\le c_1\int_0^t(1+(t-s)^{-\frac12-\frac N{2p}})e^{-\lambda_1(t-s)}c_2 M^\theta(s) e^{-\theta \eta s} \norm[\Lin]{c_0}^{1-\theta}e^{-(1-\theta)\eta s}ds\\\nn
&~~~~~~~~~~~~~~~~~~~~~~~~~~~~~~~~~~~~~~~~~~~~~~+ c_1\int_0^t(1+(t-s)^{-\frac12-\frac N{2p}})e^{-\lambda_1(t-s)}c_2 \norm[\Lin]{c_0} e^{-\eta s}ds
\end{align}
for all $t>0$.
If we multiply $e^{\eta t}$ on both sides of (\ref{5.1.1}) and let $\tilde{M}:=\mathop{\sup}\limits_{t\in(0,T)} M(t)$ for any $T\in(0,\infty)$, we find that
\begin{align}\nn
\tilde{M}\le c_5{\tilde{M}}^\theta+c_6,
\end{align}
for some $c_5,c_6>0$, thus $\tilde {M}$ is bounded, which leads to the assertion.
\end{proof}
\begin{lem}\label{lem:con:u}
Let $(n,c,u,P)$ be the classical bounded solution of (\ref{eq}).
There is $C>0$, such that
\begin{align}
\label{bdd:u}
&\norm[L^2(\Om)]{u(\cdot,t)}\le C,\\
\label{bdd:ualpha}
&\norm[L^2(\O)]{A^\alpha u(\cdot,t)}\le C\;\;
\text{ for all }t>0,\\
\label{con:ul2}
&\|u(\cdot,t)\|_{L^2(\O)}\to 0,\text{ as }t\to \infty,\\
\label{con:uinfity}
&\norm[\Lin]{u(\cdot,t)}\to 0,\text{ as } t\to\infty.
\end{align}
\end{lem}
\begin{proof}
First, (\ref{bdd:u}) and \eqref{bdd:ualpha} are immediately obtained by (\ref{con:uae}) and (\ref{con:u}), respectively. Since (\ref{interpolation}) together with (\ref{bdd:ualpha}) and (\ref{con:ul2}) immediately
 implies (\ref{con:uinfity}), it is left to prove (\ref{con:ul2}).
Testing the third equation in (\ref{eq}) with $u$ and integrating by part, we obtain
\begin{align*}
\frac12\frac{d}{dt}\into |u|^2+\into|\nabla u|^2&=\into n\nabla\phi\cdot u=\into(n-\bar{n}_0)\nabla\phi\cdot u \\
&\le \frac{\lambda_1'}{2}\into|u|^2
+\frac{1}{2\lambda_1'}\|\nabla\phi\|_{L^\infty(\Om)}^2\into |n-\bar{n}_0|^2
\end{align*}
for all $t\in(0,\infty)$. Using the Poincar\'e inequality
\begin{align}\label{ineq:u}
\frac{d}{dt}\into |u|^2+\lambda_1'\into|u|^2\le \frac{1}{\lambda_1'}\norm[\Lin]{\nabla\phi} \norm[L^2(\Om)]{n-\bar{n}_0}
\;\;\text{ for all } t>0.\end{align}
By the boundedness of $\norm[\Lin]{n(\cdot,t)}$ and ODE comparison lemma,  we conclude that $\norm[L^2(\Om)]{u(\cdot,t)}<c_1$ for some $c_1>0$ and for all $t>0$. If we apply Lemma \ref{lem:con:n}, for any $\epsilon>0$, we can find $t_0>0$ large enough satisfying 
\begin{align*}\norm[\Lin]{n(\cdot,t)-\bar{n}_0}< \frac{\lambda_1'\epsilon}{\sqrt{2
\norm[\Lin]{\nabla\phi}|\Om|}} \text{ for all } t>t_0.
\end{align*}
Again, (\ref{ineq:u}) with Gronwall's inequality implies that 
\begin{align*}
\into|u(\cdot,t)|^2&\le e^{-\lambda_1'(t-t_0)}\int|u(\cdot,t_0)|^2+\int_{t_0}^t
e^{-\lambda_1'(t-s)}\frac1{\lambda_1'}\norm[\Lin]{\nabla\phi}|\Om|\norm[\Lin]{n(\cdot,t)-\overline{n_0}}^2 ds\\
&\le e^{-\lambda_1'(t-t_0)}c_1^2+\frac{1}{(\lambda_1')^2}\norm[\Lin]{\nabla\phi}|\Om|\norm[\Lin]{n(\cdot,t)-\overline{n_0}}^2\\
&\le \epsilon^2
\end{align*}
 for all $t>\max\{t_0,\frac1{\lambda_1'}\ln e^{\frac{t_0}{\lambda_1'}}\frac{2c_1^2}{\epsilon^2}\}$.
Thus we have shown \eqref{con:ul2}.
%
%
%


%
\end{proof}

Now we are ready to prove the main result.

\begin{proof}[Proof of Theorem \ref{th1}]
The function \((n,c,u)\) obtained as the limit of \((\nep,\cep,\uep)\) is a weak solution of (\ref{eq}) by Lemma \ref{lem:isweaksol}. Moreover, its smooth regularity is guaranteed by Lemma \ref{lem:regularity2}. Hence we know that it solves (\ref{eq}) classically. The boundedness of $\norm[\Lin]{n(\cdot,t)}$ can be seen from Proposition \ref{prop:bddn} and Lemma \ref{con:nae}. Lemma \ref{lem:nc}, \ref{lem:con:nablac} imply that $\norm[W^{1,q}(\O)]{c(\cdot,t)}$ is bounded. Moreover, $u\in L^\infty((0,T);D(A^\beta))$ is asserted by (\ref{bdd:ualpha}). The continuity up to the initial time can be proven similarly as \cite[Lemma 5.8]{cao_lan}; first we prove that for $T>0$, $n_t$, $c_t$ and $u_t$ are in $L^2((0,T);(W^{1,2}(\Om))^*)$ and $L^2((0,T);(W_{0,\sigma}^{1,2}(\Om))^*)$, respectively. Then we can conclude the assertions for $c$ and $u$ by embedding. Using the continuity of $\nep$ and the uniform convergence (\ref{con:nae}), the continuity of $n$ can be done similarly as in \cite[Lemma 5.8]{cao_lan}.
Hence we have proved Theorem \ref{th1}.
\end{proof}

Now we have already shown the convergence of the solution (see Lemma \ref{lem:con:n}, Lemma \ref{lem:con:nablac} and Lemma \ref{lem:con:u}). In order to show the convergence rates are exponential, we only need to apply known result from \cite{cao_lan}, where the Navier-Stokes is considered. However, the arguments there obviously work for the Stokes case which can be seen by dropping the convective term. 
\begin{proof}[Proof of Corollary \ref{cor:con_regularity}]
From \cite[Thm 1]{cao_lan}, we know that for all $m>0$, and $\alpha,\alpha'>0$ as chosen in Corollary \ref{cor:con_regularity}, there is $p_0$, $q_0$ and $\ep_0$ such that if initial data satisfies
\begin{align}\label{init_small}
  \overline{\tilde{n}}_0=\frac1{|\O|}\intO {\tilde{n}_0}=m,\quad \norm[L^{p_0}(\O)]{ {\tilde{n}}_0- \overline{\tilde{n}}_0}\le \ep_0,\quad \norm[\Lin]{ {\tilde{c}}_0}\le\ep_0,\quad \norm[L^N(\O)]{\tilde{u}_0}\le \ep_0,
 \end{align}
the solution fulfilling
\[
\norm[L^\infty(\O)]{n(\cdot,t)-\nbar_0} \leq Ce^{-\alpha t}, \qquad \norm[L^{q_0}(\O)]{ c(\cdot,t)}\leq Ce^{-\alpha t},\qquad \norm[L^\infty(\O)]{u(\cdot,t)}\leq Ce^{-\alpha' t}\;\; \mbox{ for all } t>0.
\]
According to Lemmata \ref{lem:con:n} \ref{lem:con:nablac} and \ref{lem:con:u}, for the aforementioned $\ep_0>0$, there is $T>0$ such that
\[
 \norm[L^{p_0}(\O)]{n(\cdot,t)-\nbar_0}\le \ep_0,\quad \norm[\Lin]{ c(\cdot,t)}\le\ep_0,\quad \norm[L^N(\O)]{u(\cdot,t)}\le \ep_0\;\;\mbox{ for all }t\ge T.
\]
Let $\tilde{n}(\cdot,t)=n(\cdot,t+T)$, $\tilde{c}(\cdot,t)=c(\cdot,t+T)$, $\tilde{u}(\cdot,t)=u(\cdot,t+T)$, it is easy to see that $(\tilde{n}_0,\tilde{c}_0,\tilde{u}_0)=(\tilde{n}(\cdot,0),\tilde{c}(\cdot,0),\tilde{u}(\cdot,0))$ fulfills (\ref{init_small}), we immediately obtain the convergence rate by substituting $(\tilde{n}_0,\tilde{c}_0,\tilde{u}_0)$ into (\ref{init_small}) and uniqueness of solution for (\ref{eq}).
\end{proof}

{\small

\end{document}